\documentclass{amsart}
\usepackage{latexsym}
\usepackage{amssymb, amsbsy, amsthm, amsmath, amstext, amsopn, verbatim}
\usepackage[color,all]{xy}
\usepackage{amsfonts}
\usepackage{amscd}
\usepackage{mathrsfs}
\usepackage{verbatim}
\usepackage{xcolor}
\usepackage{tabularx}

\input xy
\xyoption{all}
\usepackage{parcolumns}

\newtheorem{thm}{Theorem} [section]
\newtheorem{lemma}[thm]{Lemma}

\newtheorem{corollary}[thm]{Corollary}
\newtheorem{prop}[thm]{Proposition}
\newtheorem{notation}[thm]{Notation}

\newtheorem{assumption}[thm]{Assumption}

\newtheorem*{rough-thm-1}{Rough Version of Vanishing Theorem}
\newtheorem*{rough-thm-2}{Rough Version of Exactness Theorem}
\newtheorem*{hkks}{Kirwan Surjectivity Problem}
\newtheorem*{main example}{Main Example}

\theoremstyle{definition}

\newtheorem*{basic convention}{Basic Conventions}

\newtheorem{defn}[thm]{Definition}

\newtheorem{example}[thm]{Example}

\newtheorem{construction}[thm]{Construction}
\newtheorem{convention}[thm]{Convention}

\theoremstyle{remark}

\newtheorem{remark}[thm]{Remark}

\begin{document}

\numberwithin{equation}{section}

\newcommand{\hs}{\mbox{\hspace{.4em}}}
\newcommand{\ds}{\displaystyle}
\newcommand{\bd}{\begin{displaymath}}
\newcommand{\ed}{\end{displaymath}}
\newcommand{\bcd}{\begin{CD}}
\newcommand{\ecd}{\end{CD}}

\newcommand{\proj}{\operatorname{Proj}}
\newcommand{\bproj}{\underline{\operatorname{Proj}}}
\newcommand{\spec}{\operatorname{Spec}}
\newcommand{\bspec}{\underline{\operatorname{Spec}}}
\newcommand{\pline}{{\mathbf P} ^1}
\newcommand{\pplane}{{\mathbf P}^2}
\newcommand{\coker}{{\operatorname{coker}}}
\newcommand{\ldb}{[[}
\newcommand{\rdb}{]]}

\newcommand{\Sym}{\operatorname{Sym}^{\bullet}}
\newcommand{\Symp}{\operatorname{Sym}}
\newcommand{\Pic}{\operatorname{Pic}}
\newcommand{\AAut}{\operatorname{Aut}}
\newcommand{\PAut}{\operatorname{PAut}}

\newcommand{\too}{\twoheadrightarrow}
\newcommand{\C}{{\mathbf C}}
\newcommand{\cA}{{\mathcal A}}
\newcommand{\cS}{{\mathcal S}}
\newcommand{\cV}{{\mathcal V}}
\newcommand{\cM}{{\mathcal M}}
\newcommand{\bA}{{\mathbf A}}
\newcommand{\aline}{\mathbb{A}^1}
\newcommand{\cB}{{\mathcal B}}
\newcommand{\cC}{{\mathcal C}}
\newcommand{\cD}{{\mathcal D}}
\newcommand{\D}{{\mathcal D}}
\newcommand{\cs}{{\mathbf C} ^*}
\newcommand{\boldc}{{\mathbf C}}
\newcommand{\cE}{{\mathcal E}}
\newcommand{\cF}{{\mathcal F}}
\newcommand{\cG}{{\mathcal G}}
\newcommand{\G}{{\mathbf G}}
\newcommand{\fg}{{\mathfrak g}}
\newcommand{\ft}{\mathfrak t}
\newcommand{\bH}{{\mathbf H}}
\newcommand{\cH}{{\mathcal H}}
\newcommand{\cI}{{\mathcal I}}
\newcommand{\cJ}{{\mathcal J}}
\newcommand{\cK}{{\mathcal K}}
\newcommand{\cL}{{\mathcal L}}
\newcommand{\baL}{{\overline{\mathcal L}}}
\newcommand{\M}{{\mathcal M}}
\newcommand{\bM}{{\mathbf M}}
\newcommand{\bm}{{\mathbf m}}
\newcommand{\cN}{{\mathcal N}}
\newcommand{\theo}{\mathcal{O}}
\newcommand{\cP}{{\mathcal P}}
\newcommand{\cR}{{\mathcal R}}
\newcommand{\boldp}{{\mathbf P}}
\newcommand{\boldq}{{\mathbf Q}}
\newcommand{\bbL}{{\mathbf L}}
\newcommand{\cQ}{{\mathcal Q}}
\newcommand{\cO}{{\mathcal O}}
\newcommand{\Oo}{{\mathcal O}}
\newcommand{\OX}{{\Oo_X}}
\newcommand{\OY}{{\Oo_Y}}
\newcommand{\otY}{{\underset{\OY}{\ot}}}
\newcommand{\otX}{{\underset{\OX}{\ot}}}
\newcommand{\cU}{{\mathcal U}}
\newcommand{\cX}{{\mathcal X}}
\newcommand{\cW}{{\mathcal W}}
\newcommand{\boldz}{{\mathbf Z}}
\newcommand{\cZ}{{\mathcal Z}}
\newcommand{\qgr}{\operatorname{qgr}}
\newcommand{\gr}{\operatorname{gr}}
\newcommand{\coh}{\operatorname{coh}}
\newcommand{\End}{\operatorname{End}}
\newcommand{\Hom}{\operatorname{Hom}}
\newcommand{\uHom}{\underline{\operatorname{Hom}}}
\newcommand{\uHomY}{\uHom_{\OY}}
\newcommand{\uHomX}{\uHom_{\OX}}
\newcommand{\Ext}{\operatorname{Ext}}
\newcommand{\bExt}{\operatorname{\bf{Ext}}}
\newcommand{\Tor}{\operatorname{Tor}}

\newcommand{\inv}{^{-1}}
\newcommand{\airtilde}{\widetilde{\hspace{.5em}}}
\newcommand{\airhat}{\widehat{\hspace{.5em}}}
\newcommand{\nt}{^{\circ}}
\newcommand{\del}{\partial}

\newcommand{\supp}{\operatorname{supp}}
\newcommand{\GK}{\operatorname{GK-dim}}
\newcommand{\hd}{\operatorname{hd}}
\newcommand{\id}{\operatorname{id}}
\newcommand{\res}{\operatorname{res}}
\newcommand{\lrar}{\leadsto}
\newcommand{\im}{\operatorname{Im}}
\newcommand{\HH}{\operatorname{H}}
\newcommand{\TF}{\operatorname{TF}}
\newcommand{\Bun}{\operatorname{Bun}}
\newcommand{\Hilb}{\operatorname{Hilb}}
\newcommand{\Fact}{\operatorname{Fact}}
\newcommand{\F}{\mathcal{F}}
\newcommand{\nthord}{^{(n)}}
\newcommand{\Aut}{\underline{\operatorname{Aut}}}
\newcommand{\Gr}{\operatorname{Gr}}
\newcommand{\Fr}{\operatorname{Fr}}
\newcommand{\GL}{\operatorname{GL}}
\newcommand{\gl}{\mathfrak{gl}}
\newcommand{\SL}{\operatorname{SL}}
\newcommand{\ff}{\footnote}
\newcommand{\ot}{\otimes}
\def\Ext{\operatorname {Ext}}
\def\Hom{\operatorname {Hom}}
\def\Ind{\operatorname {Ind}}
\def\bbZ{{\mathbb Z}}

\newcommand{\nc}{\newcommand}
\newcommand{\on}{\operatorname}
\nc{\cont}{\on{cont}}
\nc{\rmod}{\on{mod}}
\nc{\Mtil}{\widetilde{M}}
\nc{\wb}{\overline}
\nc{\wt}{\widetilde}
\nc{\wh}{\widehat}
\nc{\sm}{\setminus}
\nc{\mc}{\mathcal}
\nc{\mbb}{\mathbb}
\nc{\Mbar}{\wb{M}}
\nc{\Nbar}{\wb{N}}
\nc{\Mhat}{\wh{M}}
\nc{\pihat}{\wh{\pi}}
\nc{\JYX}{\cJ_{Y\leftarrow X}}
\nc{\phitil}{\wt{\phi}}
\nc{\Qbar}{\wb{Q}}
\nc{\DYX}{\D_{Y\leftarrow X}}
\nc{\DXY}{\D_{X\to Y}}
\nc{\dR}{\stackrel{\bbL}{\underset{\D_X}{\ot}}}
\nc{\Winfi}{\cW_{1+\infty}}
\nc{\K}{{\mc K}}
\nc{\unit}{{\bf \on{unit}}}
\nc{\boxt}{\boxtimes}
\nc{\xarr}{\stackrel{\rightarrow}{x}}
\nc{\Cnatbar}{\overline{C}^{\natural}}
\nc{\oJac}{\overline{\on{Jac}}}
\nc{\gm}{{\mathbf G}_m}
\nc{\Loc}{\on{Loc}}
\nc{\Bm}{\operatorname{Bimod}}
\nc{\lie}{{\mathfrak g}}
\nc{\lb}{{\mathfrak b}}
\nc{\lien}{{\mathfrak n}}
\nc{\e}{\epsilon}
\nc{\eu}{\mathsf{eu}}

\nc{\Gm}{{\mathbb G}_m}
\nc{\Gabar}{\wb{\G}_a}
\nc{\Gmbar}{\wb{\G}_m}
\nc{\PD}{{\mathbb P}_{\D}}
\nc{\Pbul}{P_{\bullet}}
\nc{\PDl}{{\mathbb P}_{\D(\lambda)}}
\nc{\PLoc}{\mathsf{MLoc}}
\nc{\Tors}{\on{Tors}}
\nc{\PS}{{\mathsf{PS}}}
\nc{\PB}{{\mathsf{MB}}}
\nc{\Pb}{{\underline{\operatorname{MBun}}}}
\nc{\Ht}{\mathsf{H}}
\nc{\bbH}{\mathbb H}
\nc{\gen}{^\circ}
\nc{\Jac}{\operatorname{Jac}}
\nc{\sP}{\mathsf{P}}
\nc{\sT}{\mathsf{T}}
\nc{\bP}{{\mathbb P}}
\nc{\otc}{^{\otimes c}}
\nc{\Det}{\mathsf{det}}
\nc{\PL}{\on{ML}}
\nc{\sL}{\mathsf{L}}

\nc{\ml}{{\mathcal S}}
\nc{\Xc}{X_{\on{con}}}
\nc{\Z}{{\mathbb Z}}
\nc{\resol}{\mathfrak{X}}
\nc{\map}{\mathsf{f}}
\nc{\gK}{\mathbb{K}}
\nc{\bigvar}{\mathsf{W}}
\nc{\Tmax}{\mathsf{T}^{md}}

\nc{\Cpt}{\mathbb{P}}
\nc{\pv}{e}

\nc{\Qgtr}{Q^{\on{gtr}}}
\nc{\algtr}{\alpha^{\on{gtr}}}
\nc{\Ggtr}{\mathbb{G}^{\on{gtr}}}
\nc{\chigtr}{\chi^{\on{gtr}}}

\newcommand{\la}{\langle}
\newcommand{\ra}{\rangle}
\newcommand{\fm}{\mathfrak m}
\newcommand{\Ms}{\mathfrak M}
\newcommand{\Ma}{\mathfrak M_0}
\newcommand{\ML}{\mathfrak L}
\newcommand{\ev}{\textit{ev}}
\nc{\thetagtr}{\theta^{\on{gtr}}}

\numberwithin{equation}{section}

\title{Kirwan Surjectivity for Quiver Varieties}

\author{Kevin McGerty}
\address{Mathematical Institute\\University of Oxford\\Oxford OX1 3LB, UK}
\email{mcgerty@maths.ox.ac.uk}
\author{Thomas Nevins}
\address{Department of Mathematics\\University of Illinois at Urbana-Champaign\\Urbana, IL 61801 USA}
\email{nevins@illinois.edu}

%\date{\today}

\begin{abstract}
For algebraic varieties defined by hyperk\"ahler or, more generally, algebraic symplectic reduction, it is a long-standing question whether the ``hyperk\"ahler Kirwan map'' on cohomology is surjective.  We resolve this question in the affirmative for Nakajima quiver varieties.  We also establish similar results for other cohomology theories and for the derived category.  Our proofs use only classical topological and geometric arguments.  
\end{abstract}

\maketitle

\section{Introduction}
Suppose $M$ is a complex algebraic variety with the action of a complex algebraic group $G$, yielding a quotient stack/equivariant space $\resol = M/G$; or more generally $\resol$ is any complex algebraic stack.  Often $\resol$ has one or more natural open sets $\resol^{\on{ss}}$---typically defined via geometric invariant theory (GIT)---that are smooth algebraic varieties; thus, when $\resol = M/G$, we have $\resol^{\on{s}} = M^{\on{ss}}/G$ where $G$ acts freely on $M^{\on{ss}}$.  Fixing such an open subset $i: \resol^{\on{ss}}\hookrightarrow \resol$, one has the following problem.
\begin{hkks}
When is the pullback map
\begin{equation}\label{hkks-eq}
H^*(\resol) \xrightarrow{i^*} H^*(\resol^{\on{ss}})
\end{equation}
surjective?
\end{hkks}
\begin{convention} 
Except when noted otherwise, $H^*$ means singular cohomology with ${\mathbb Z}$ coefficients.
\end{convention}
\noindent
When $\resol$ itself is smooth and $\resol^{\on{ss}}$ is defined by GIT, classical Morse-theoretic results of Atiyah-Bott and Kirwan show that the ``Kirwan map'' \eqref{hkks-eq} is surjective.  
Significant recent attention focuses on the case when $\resol$ is a singular, but algebraic symplectic or even hyperk\"ahler, stack: typically, letting $Z$ be a smooth $G$-variety with algebraic moment map $\mu: T^*Z\rightarrow \mathfrak{g}^* = \on{Lie}(G)^*$, we have $\resol  = \mu\inv(0)/G$ and $\resol^{\on{ss}} = \mu\inv(0)^{\on{ss}}/G$ for a choice of GIT stability.

\vspace{.5em}

This paper resolves the Kirwan Surjectivity Problem when $\resol^{\on{ss}}$ is a Nakajima quiver variety.  

\vspace{.5em}

Thus, let $Q= (I,\Omega)$ be a quiver and $\mathbf v, \mathbf w \in \mathbb{Z}_{\geq 0}^I$ vectors with $\mathbf{w}\neq \mathbf 0$. 
 Following Nakajima \cite{Nakajima Duke, NakajimaJAMS}, these data yield (notation as in Section \ref{sec:basicsofquivers}):
\begin{enumerate}
\item a finite-dimensional complex vector space $\mathbb{M} = \mathbb{M}({\mathbf v},{\mathbf w})$, with
\item the linear action of the complex group $\mathbb{G} = \prod_i GL_{v_i}$, and
\item a (complex) moment map $\mu: \mathbb{M} \longrightarrow \on{Lie}(\mathbb{G})^*$.
\end{enumerate}
Fix a nondegenerate stability condition $\theta$ (Definition \ref{nondegen}) in the sense of GIT---for example the one used in \cite{Nakajima Duke, NakajimaJAMS}---with stable locus $\mu\inv(0)^{\on{ss}} = \mu\inv(0)^s \subset \mathbb{M}^s$.  The $\mathbb{G}$-action on $\mathbb{M}^s$ is free, and the quotient $\Ms = \Ms(\mathbf v, \mathbf w) := \mu\inv(0)^s/\mathbb{G}$ is the Nakajima quiver variety associated to $Q, \mathbf{v}, \mathbf{w}, \theta$.
\begin{thm}\label{main thm}
Let $\Ms(\mathbf v, \mathbf w)$ be a smooth Nakajima quiver variety. 
Then the Kirwan map 
\bd
H^*_{\mathbb{G}}(\on{pt}) \cong H^*_{\mathbb{G}}(\mu\inv(0))\longrightarrow H^*_{\mathbb{G}}(\mu\inv(0)^s) = H^*\big(\Ms(\mathbf v, \mathbf w)\big)
\ed
is surjective.  Thus, $H^*\big(\Ms(\mathbf v, \mathbf w)\big)$ is generated by tautological classes.
\end{thm}
\noindent
We note that $H^*_{\mathbb{G}}(\on{pt})$ is a polynomial ring (in the tautological classes of the theorem).  Theorem \ref{main thm} extends to many other cohomology theories, including complex $K$-theory and elliptic cohomology.\ff{We have in mind Grojnowski's equivariant elliptic cohomology \cite{Grojnowski}, since it seems to be the only theory currently documented; though the same arguments apply to any theory with standard formal properties.}
\begin{thm}\label{complex-oriented}
Assume that  $E^*(\on{pt})$ is concentrated in even degrees.
\begin{enumerate}
\item The map $E^*(B\mathbb{G})\rightarrow   E^*(\Ms)$ is surjective.  
\item If $E$ is complex-oriented,  then $E^*(\Ms)$ is generated as an $E^*(\on{pt})$-algebra by Chern classes of tautological bundles.  
\end{enumerate}
\end{thm}
\begin{corollary}
The natural maps $K_{\mathbb{G}}^*(\on{pt})\rightarrow K^*(\Ms)$ and $\on{Ell}_{\mathbb{G}}^*(\on{pt}) \rightarrow \on{Ell}^*(\Ms)$ are surjective.
\end{corollary}

Furthermore, if $\mathbb{T}$ is an algebraic torus acting on $\Ms$, then, because $H^*(\Ms)$ is evenly graded, the Leray spectral sequence for $H^*_{\mathbb{T}}(\Ms)$ degenerates, showing that $H^*(\Ms) = H^*_{\mathbb{T}}(\Ms) \otimes_{H^*_{\mathbb{T}}(\on{pt})} H^*(\on{pt})$.  Via Theorem \ref{main thm} and the Nakayama Lemma for graded rings, we conclude:
\begin{corollary}\label{T-equivariant}
For a torus $\mathbb{T}$ acting on $\Ms$, the map $H^*_{\mathbb{G}\times\mathbb{T}}(\on{pt})\rightarrow H^*_{\mathbb{T}}(\Ms)$ is surjective.
\end{corollary} 
In particular, the expectation expressed in Section 2.2.2 of \cite{AO}---that their map (9) is an embedding near the origin of $\mathscr{E}_{\mathsf{T}}$---follows.  We note the applicability of the above results in other, similar contexts (cf. \cite{MO}).  Analogues of Corollary \ref{T-equivariant} can also be proven for $K$-theory and elliptic cohomology equivariant with respect to a torus $\mathbb{T}$ or more general ``flavor symmetries'' of $\Ms$.

Our method also yields the following.  
\begin{thm}\label{derived cat}
Let $D(\Ms)$ denote the unbounded quasicoherent derived category of $\Ms$, and $D_{\on{coh}}(\Ms)$ its bounded coherent subcategory.
\begin{enumerate}
\item The category $D(\Ms)$ is generated by tautological bundles.  
\item There is a finite list of tautological bundles from which every object of  $D_{\on{coh}}(\Ms)$ is obtained by finitely many applications of (i) direct sum, (ii) cohomological shift, and (iii) cone.
\end{enumerate}
\end{thm}
We note that the second assertion of Theorem \ref{derived cat} is {\em not} simply a formal consequence of the first, since we do {\em not} include taking direct summands (i.e., retracts) among the operations (i)-(iii).  Results related to Theorem \ref{derived cat} appear in \cite{HL}.

We mention one further application of Theorem \ref{main thm} (that will be readily apparent to experts).
\begin{corollary}[Assumption 5.13 of \cite{BDMN}]
Let $\mathfrak{g} = \on{Lie}(\mathbb{G})$, and $Z:= Z(\mathfrak{g})^*\subset \mathfrak{g}^*$ denote the dual of the center.  Consider the family $\mathfrak{M} = \mu\inv(Z)^{\on{ss}}/\mathbb{G}\longrightarrow Z$ of Hamiltonian reductions.  Then the Duistermaat-Heckman map for this family is surjective.  In particular, the family of Hamiltonian reductions $\mathfrak{M}\rightarrow Z$ provides a versal Poisson deformation of the Nakajima quiver variety $\Ms$.
\end{corollary}
\noindent

   Cases of Kirwan surjectivity for quivers of finite and affine Dynkin type, and for star-shaped quivers, have previously been established 
(see  \cite{Vasserot, Webster, KoderaNaoi, SVV, FisherRayan}) by different techniques; and for moduli of $GL_n$-Higgs bundles by Markman \cite{Markman}.

Here is a sketch of the strategy used to prove Theorem \ref{main thm}.  
\begin{enumerate}
\item We compactify $\Ms$ to a projective variety $\overline{\Ms}$ by an explicit quiver construction.
\item We identify the class of the graph $\Gamma$ of the inclusion $i: \Ms\hookrightarrow\overline{\Ms}$ 
in $H^*(\Ms\times\overline{\Ms})$
as a Chern class of a complex built from external tensor products of tautological bundles on $\Ms\times\overline{\Ms}$.
\item Purely topological arguments allow us to conclude (Section \ref{sec:tools}) that the Chern classes of the tautological bundles generate the cohomology of $\Ms$.
\end{enumerate}
We emphasize that the overall strategy is not new: see \cite{Nakajima Duke} (and \cite{Markman}).  The new ingredient here is the particular choice of {\em modular} compactification $\overline{\Ms}$.  
Hiding behind our approach to Theorem \ref{main thm} and our other results is, in fact, a general pattern (that experts may already discern here) for moduli spaces in noncommutative geometry---that is, moduli of objects in certain categories.  
The general story
 will be worked out in a forthcoming paper \cite{McNkirwan}.
Nonetheless, it seemed desirable to us to present the results for quiver varieties separately.  Indeed, on the one hand, the proof of Theorem \ref{main thm} can be made completely classical and explicit for quivers, in a way that avoids any categorical yoga or abstraction (and thus will be of independent interest to some readers).  On the other hand, we also obtain sharper results for quiver varieties than seem to be  easily achievable in a completely general context.      

\begin{convention}
Throughout the paper, all varieties, groups, etc. are defined over $\mathbb{C}$.
\end{convention}

\subsection*{Acknowledgments}
We thank G. Bellamy for comments on a draft.  The first author was supported by EPSRC programme grant EI/I033343/1.  The second author was supported by NSF grants DMS-1159468 and DMS-1502125.

\section{Topology of Compactifications}\label{sec:tools}
Throughout the paper, we use $H^*(X)$, with no further decorations indicating coefficients, to denote cohomology with $\mathbb{Z}$-coefficients, and $H^{\on{BM}}_*(X)$ to denote Borel-Moore homology with $\mathbb{Z}$-coefficients; if $X$ is smooth, there is a canonical isomorphism $H^*(X) \cong H^{\on{BM}}_*(X)$.  
\subsection{Pushforwards and the Projection Formula}  
Suppose $f: X\rightarrow Y$ is a proper morphism of relative dimension $d$ of smooth, connected varieties (or Deligne-Mumford stacks).  Then there is a pushforward, or Gysin, map $f_*: H^*(X)\rightarrow H^{*-d}(Y)$.

The Gysin map satisfies the projection formula: for classes $c\in H^*(X), c'\in H^*(Y)$, we have
\begin{equation}\label{projection formula}
f_*(c\cup f^*c') = f_*(c)\cup c'.
\end{equation}
Moreover, if $f: X\rightarrow Y$ is a closed immersion, then 
\begin{equation}\label{pullpush}
f_*f^*c = c\cup [X],
\end{equation}
where $[X]$ denotes the fundamental class of $X$ in Borel-Moore homology (which is canonically isomorphic to cohomology since $Y$ is smooth).

\subsection{K\"unneth Components and Images Under Pullback}  
Suppose $C\in H^*(X\times Y)$ is a cohomology class and that the K\"unneth formula $H^*(X\times Y)\cong H^*(X)\otimes H^*(Y)$ holds.\ff{This is true when one of $X$, $Y$ is a Nakajima quiver variety: Nakajima proves that its cohomology is free abelian.}  We may write 
\begin{equation}\label{Kunneth}
C = \sum x_i\otimes y_i \; \text{with $x_i\in H^*(X)$, $y_i\in H^*(Y)$}.
\end{equation}  The classes $x_i$, $y_i$ are the {\em left-hand, respectively right-hand, K\"unneth components} of $C$ with respect to the decomposition \eqref{Kunneth}; they are not independent of the choice of decomposition \eqref{Kunneth}.

Now suppose that $f: X\rightarrow Y$ is a morphism from a smooth variety $X$ to a smooth, proper variety $Y$.  Let $\Gamma_f\subset X\times Y$ be the graph of the map.
\begin{prop}\label{prop:image}
The image of $f^*: H^*(Y)\rightarrow H^*(X)$ is contained in the span of the K\"unneth components of $[\Gamma_f]$ with respect to $X$ (and any decomposition as in \eqref{Kunneth}). 
\end{prop}
\begin{proof} 
Write $\xymatrix{X &  X\times Y \ar[l]_{p_X} \ar[r]^{p_Y} & Y}$ for the projections and, abusively, $\Gamma_f: X\rightarrow X \times Y$ for both the graph immersion and its image.    Write $p_*: Y\rightarrow \on{Spec}(\mathbb{C})$ for the projection to a point. Then $(p_X)_*$ exists since $Y$ is proper, and 
\bd
f^*d = (p_X)_* (\Gamma_f)_*\Gamma_f^*p_Y^* d =  (p_X)_*([\Gamma_f]\cup p_Y^*d) = \sum_i (p_X)_*\big[(p_X^* x_i\cup p_Y^*y_i)\cup p_Y^*d\big]
= \sum_i x_i \cup p_*(y_i\cup d).
\ed
This proves the claim.
\end{proof}

\subsection{Resolution of a Graph}
Again suppose that $f: X\rightarrow Y$ is a morphism from a smooth variety to an irreducible projective variety, with graph $\Gamma\subset X\times Y$.    We assume $f(X)\subset Y^{\on{sm}}$, the smooth locus of $Y$.  We consider the situation in which $f^*: H^*(Y)\rightarrow H^*(X)$ is surjective.
\begin{example}
If $H^{\on{BM}}_*(X,\mathbb{Z})\cong H^*(X,\mathbb{Z})$ is generated by algebraic cycles and $X\rightarrow Y$ is an open immersion, then $H^*(Y, \mathbb{Z})\rightarrow H^*(X,\mathbb{Z})$ is surjective.
\end{example}
\begin{remark}\label{quiver remark}
For a Nakajima quiver variety $\Ms$, $H^*(\Ms,\mathbb{Z})$ is generated by algebraic cycles by Theorem 7.3.5 of \cite{NakajimaJAMS}.
\end{remark}

Continuing with the above situation, let $\wt{Y}$ be a resolution of singularities; since $f(X)$ does not intersect the singular locus of $Y$, $f$ lifts canonically to a morphism $\wt{f}: X\rightarrow \wt{Y}$ and the preimage of $\Gamma_f$ in $X\times\wt{Y}$ is $\Gamma_{\wt{f}}$.  
\begin{prop}\label{prop:resolution}
Suppose that 
\begin{equation}\label{R summary}
R: \hspace{1em} \bigoplus_j\cE_j^{-1}\boxtimes \cF_j^{-1} \longrightarrow \bigoplus_j\cE_j^0\boxtimes \cF_j^0\longrightarrow \bigoplus_j\cE_j^1\boxtimes\cF_j^1
\end{equation}
is a complex of vector bundles on $X\times Y$ with the following properties.
\begin{enumerate}
\item $\mathcal{H}^1(R) = 0$, $\mathcal{H}^{-1}(R)=0$, and $\mathcal{H}:= \mathcal{H}^0(R)$ is a vector bundle on $X\times Y$.
\item $\on{rk}(\mathcal{H}) = d:=\dim(Y)$.
\item $s\in H^0(X\times Y, \mathcal{H})$ is a section with scheme-theoretic zero locus $Z(s) = \Gamma$.
\end{enumerate}
Letting $\wt{Y}\rightarrow Y$ be a resolution of singularities, write
\bd 
\wt{R}: \hspace{1em} \bigoplus_j\cE_j^{-1}\boxtimes \wt{\cF}_j^{-1} \longrightarrow \bigoplus_j\cE_j^0\boxtimes \wt{\cF}_j^0\longrightarrow \bigoplus_j\cE_j^1\boxtimes\wt{\cF}_j^1
\ed
for the pullback of $R$ to $X\times\wt{Y}$ and $\wt{\mathcal{H}} = \mathcal{H}^0(\wt{R})$.  Then:
\begin{enumerate}
\item[(i)] $c_d(\wt{\mathcal{H}}) = [\Gamma_{\wt{f}}]$ in $X\times\wt{Y}$.
\item[(ii)] The Chern classes of $\wt{\mathcal{H}}$ are polynomials, with integer coefficients, in the Chern classes of the bundles $\cE_j^{\ell}$ and $\wt{\cF}_j^{\ell}$.
\item[(iii)] The image of the map $H^*(\wt{Y},\mathbb{Z}) \rightarrow H^*(X,\mathbb{Z})$ is contained in the span of the Chern classes of the bundles $\cE_j^\ell$.
\end{enumerate}
\end{prop}
\begin{proof}
(i) It is standard that if the zero locus of a section of a vector bundle $\wt{\mathcal{H}}$ of rank $d$ has codimension $d$---in which case it is a local complete intersection subscheme---then its fundamental class equals $c_d(\wt{\mathcal{H}})$.

(ii) By the additivity of Chern classes, we have  
\bd
c(\wt{\mathcal{H}}) = \prod_j c(\cE_j^0\boxtimes\wt{\cF}_j^0) \prod_j c(\cE_j^{-1}\boxtimes\wt{\cF}_j^{-1})\inv \prod_j c(\cE_j^1\boxtimes\wt{\cF}_j^1)\inv.
\ed
The inverses of the total Chern classes are the total Segre classes, which are known to be polynomials, with integer coefficients, in the Chern classes: see Chapter 5 of \cite{Fulton}.  Moreover, the Chern classes of $\cE_j^\ell\boxtimes\wt{\cF}_j^\ell$ are also polynomials (with integer coefficients) in the Chern classes of $\cE_j^\ell$ and $\wt{\cF}_j^\ell$: see Example 14.5.2 of \cite{Fulton}.\footnote{This is, however, abstractly clear: the Chern classes are pulled back along the composite $X\times\wt{Y}\rightarrow BGL(\on{rk}(\cE_j^\ell))\times BGL(\on{rk}(\wt{\cF}_j^\ell))\xrightarrow{\otimes}  BGL\big(\on{rk}(\cE_j^\ell) \cdot \on{rk}(\wt{\cF}_j^\ell)\big)$, hence are polynomials in the cohomology classes generating $H^*\big(BGL(\on{rk}(\cE_j^\ell))\times BGL(\on{rk}(\wt{\cF}_j^\ell))\big)$.}  

(iii) By parts (i) and (ii), the class $[\Gamma_{\wt{f}}]$ has a K\"unneth decomposition \eqref{Kunneth} whose left-hand components are integer polynomials in the Chern classes of the bundles $\cE_j^\ell$.  Assertion (iii) is now immediate from Proposition \ref{prop:image}.
\end{proof}
\begin{corollary}\label{cohomology generation}
Suppose that $\Ms$ is a smooth Nakajima quiver variety and $\Ms\hookrightarrow\overline{\Ms}$ is an open immersion in a projective variety.  If the graph $\Gamma$ of the immersion can be written as the zero locus $Z(s)$ of a section $s\in H^0(\Ms\times\overline{\Ms},\mathcal{H})$ of a vector bundle $\mathcal{H}$ as in Proposition \ref{prop:resolution}, then $H^*(\Ms,\mathbb{Z})$ is generated by the Chern classes of the bundles $\cE_j^\ell$.
\end{corollary}
\begin{proof}
As explained above, $H^*(\Ms,\mathbb{Z})$ is known to be generated by algebraic cycles; hence  (cf. Proposition 1.8 of \cite{Fulton}) for any projective compactification $\overline{\Ms}$ the restriction map 
$H^*(\overline{\Ms},\mathbb{Z})\rightarrow H^*(\Ms,\mathbb{Z})$ is surjective.  The assertion is now immediate from Proposition \ref{prop:resolution}.
\end{proof}

\section{Quiver Varieties}\label{sec:quiver-var}
\subsection{Basics of Quivers}\label{sec:basicsofquivers}
Let $(I,E)$ be an undirected graph with with vertex set $I$ and edge set $E$.  Following Nakajima \cite{Nakajima Duke, NakajimaJAMS}, we let $H$ denote the set of pairs of an edge with an orientation; thus $H$ comes with source and target maps $s,t: H\rightarrow I$.  Given $h\in H$, we let $\overline{h}$ denote the same edge with opposite orientation, so $s(\overline{h})= t(h)$ and $t(\overline{h}) = s(h)$.  

Next, fix a preferred orientation for each edge: in other words, fix a decomposition $H = \Omega\sqcup\overline{\Omega}$ where
$\overline{\Omega} = \{\overline{h}\; |\; h\in\Omega\}$.  We let $Q = (I,\Omega)$ denote the quiver, i.e., the finite directed graph, with vertices $I$ and arrows $\Omega$; then $Q^{\on{dbl}} = (I, H)$ is the associated {\em doubled quiver}.    
We define a function 
\bd
\epsilon: H\longrightarrow\{\pm 1\} \hspace{1em}\text{by}\hspace{1em} \epsilon(h) = \begin{cases} 1 &  \text{if $h\in\Omega$},\\ -1 & \text{if $h\in\overline{\Omega}$}.\end{cases}
\ed 
The {\em preprojective algebra} is the quotient $\ds\Pi^0(Q) = kQ^{\on{dbl}}/\Big(\sum_{h\in H} \epsilon(h)\overline{h}h\Big)$ of the path algebra $kQ^{\on{dbl}}$ of the doubled quiver.  The relation $\ds\sum_{h\in H} \epsilon(h)\overline{h}h = 0$ is the {\em preprojective relation}.  

If $V_\bullet$ is an $I$-graded vector space, then $\ds\on{Rep}(Q,V_\bullet) = \bigoplus_{h\in\Omega} \Hom(V_{s(h)}, V_{t(h)})$.  When $\mathbf{v}\in\mathbb{Z}_{\geq 0}^I$ and $V_i = \mathbb{C}^{\mathbf{v}_i}$ for all $i\in I$, we write $\on{Rep}(Q,\mathbf{v}) = \on{Rep}(Q,V_\bullet)$.

Let $\mathbf v = (v_i)_{i\in I}, \mathbf w = (w_i)_{i\in I}$ be dimension vectors, and $V_i, W_i (i\in I)$ be complex vector spaces with $\on{dim}(V_i) = v_i, \dim(W_i) = w_i$; here $W_i$ are the {\em framing vector spaces}.  Given pairs ${\mathbf v}^1, {\mathbf w}^1$ and ${\mathbf v}^2, {\mathbf w}^2$ and vector spaces $V_i^j, W_i^j$ ($j=1,2$) as above, let
\bd
L(V^1, V^2) = \bigoplus_{i\in I} \Hom(V_i^1, V_i^2), \hspace{3em}  E(V^1, V^2) = \bigoplus_{h\in H} \Hom(V^1_{s(h)}, V^2_{t(h)}).
\ed
One has obvious actions of $L(V^1,V^2)$ on $L(V^2,V^3)$, $L(V^3,V^1)$, $E(V^2,V^3)$, and $E(V^3,V^1)$.

On p. 520 of \cite{Nakajima Duke}, Nakajima defines a bilinear multiplication 
\bd
E(V^2, V^3)\times E(V^1, V^2) \rightarrow L(V^1, V^3),  \hspace{2em} \text{by}
\ed 
\bd
\ds (C,B)\mapsto CB = \left(\sum_{t(h) = k} C_hB_{\overline{h}}\right)_k \in L(V^1, V^3).
\ed

Now, fixing $\mathbf{v}, \mathbf{w}$ and collections of vector spaces $(V_i), (W_i)$ as above, let
\bd
\mathbb{M} = \mathbb{M}(\mathbf{v},\mathbf{w}) = E(V,V)\oplus L(W,V)\oplus L(V,W).
\ed
We write $[B,i,j]$ for an element of $\mathbb{M}$.
The group 
\bd
\mathbb{G} = \mathbb{G}(\mathbf{v})= \prod_i GL(V_i)\cong \prod_i GL_{v_i}
\ed
 acts linearly on $\mathbb{M}$ in the obvious way.  There is a canonical moment map, coming from the identification of $\mathbb{M}$ as a cotangent bundle to a linear space, written $\mu: \mathbb{M}\longrightarrow \on{Lie}(\mathbb{G})^*$.  

\subsection{Crawley-Boevey's Construction}\label{sec:CB quiver}
Suppose $Q=(I,\Omega)$ is a quiver with dimension vectors $\mathbf{v}, \mathbf{w}$ as above.  To such data, Crawley-Boevey associates \cite[Section~1]{CB} a new quiver, that we will denote by $Q^{\on{CB}}$.  It has vertex set $I^{\on{CB}} = I\cup \{\infty\}$, and oriented arrows 
\bd
\Omega^{\on{CB}} = \Omega\cup \big\{a_{(i,j)} \; |\; s(a_{i,j}) = \infty, t(a_{i,j}) = i,\; i\in I, j\in \{1,\dots, \mathbf{w}_i\}\big\}.
\ed
In other words, we add $\mathbf{w}_i$-many arrows from $\infty$ to $i$.  Let $\alpha\in {\mathbb Z}_{\geq 0}^{I^{\on{CB}}}$ be the dimension vector for $Q^{\on{CB}}$ that equals $\mathbf{v}_i$ at $i\in I$ and $1$ at the vertex $\infty$.  Then $\mathbb{M}(\mathbf{v},\mathbf{w}) = T^*\on{Rep}(Q^{\on{CB}},\alpha)$.   Also the natural homomorphism  $\mathbb{G}\rightarrow \ds G(\alpha) := \prod_{i\in I^{\on{CB}}}  GL(\alpha_i) / \Gm$ (where $\Gm$ is the diagonal multiplicative group) is an isomorphism, making the identification of $\mathbb{M}(\mathbf{v},\mathbf{w})$ with $T^*\on{Rep}(Q^{\on{CB}},\alpha)$
equivariant.   It is immediate that the two canonical moment maps coincide.  

\subsection{Semistability and Stability for Quiver Representations}\label{sec:semistability}
Fix a quiver $Q= (I,\Omega)$ with dimension vector $\mathbf{v}$.
Let $\mathbb{G} = \prod_i GL(\alpha_i)$ denote the group detemined by $Q$.
We write $\beta\leq \alpha$ for a dimension vector $\beta$ if $\beta_i\leq \alpha_i$ for all $i\in I$, and $\beta<\alpha$ if $\beta\leq \alpha$ and $\beta\neq\alpha$.

Following \cite{King}, given a character $\chi:\mathbb{G}\rightarrow\Gm$,
write 
\bd
\chi\big((g_i)_{i\in I}\big) = \prod_{i\in I} \det(g_i)^{\theta_i} \hspace{1em} \text{and} \hspace{1em} \theta = (\theta_i)_{i\in I}.
\ed
Given an $I$-graded vector space $(M_i)_{i\in I}$, we define $\delta_i(M) =\dim(M_i)$, and thus define 
\bd
\theta = \sum_i \theta_i \delta_i  \hspace{1em}\text{so that} \hspace{1em} \theta(M) = \sum_i\theta_i\dim(M_i).
\ed

  Associated to $\chi$ one gets a corresponding notion of GIT semistability as in \cite{King}.  
  In particular, by Proposition 3.1 of \cite{King}, if $V$ is a representation of $Q$ for which $\theta(V) = 0$, then $V$ is $\chi$-semistable, respectively stable, if and only if for every nonzero proper subrepresentation $M\subset V$, we have
  \bd
\theta(M) \geq 0, \hspace{2em}
\text{respectively $\theta(M)>0$.}
\ed  
\begin{defn}\label{nondegen}
We will call the semistability condition $\theta$ a {\em nondegenerate} stability (with respect to $\alpha$) if, for every nonzero dimension vector $\beta < \alpha$ we have $\theta(\beta)\neq 0$.  
\end{defn}
\begin{remark}
If $\theta$ is a nondegenerate stability condition, then semistability and stability coincide.  The converse is treated in \cite[Theorem~2.8]{Nakajima Duke}.
\end{remark}

Now suppose that $Q_0 = (I_0,\Omega_0)$ is a quiver with dimension vector $\mathbf{v}$ and framing vector $\mathbf{w}$ and that $Q = Q_0^{\on{CB}} = (I,\Omega)$ is the associated Crawley-Boevey quiver, with dimension vector $\alpha$ so that $\alpha_\infty = 1$ and $\alpha|_{I_0} = \mathbf{v}$.  
Write $\mathbb{G}_0 = \prod_{i\in I_0}GL(\mathbf{v}_i)$ and $\mathbb{G} = \prod_i GL(\alpha_i)$, so that $\mathbb{G} = \mathbb{G}_0\times\Gm$.  
Given any character $\chi_0: \mathbb{G}_0\rightarrow \Gm$,
$\chi_0(g_i) = \prod_{i\in I_0} \det(g_i)^{\theta_i}$,
 let $\delta: \Gm\rightarrow \mathbb{G}_0$ be the diagonal $\Gm$ and write $\chi_0(\delta(z)) = z^d$.  We get a character $\chi: \mathbb{G} \rightarrow \Gm$ by $\chi(g, z) = \chi_0(g)z^{-d}$; we slightly abusively write
 $\ds\theta = \sum_{i\in I_0}(\theta_0)_i\delta_i -d \delta_{\infty}$.  Then $\chi$ is trivial on the diagonal $\Gm$ in $\mathbb{G}_0\times\Gm$, and thus it factors through a character of $G(\alpha): = \mathbb{G}/\Gm$, which obviously agrees with $\chi_0$ under the isomorphism $\mathbb{G}_0\rightarrow G(\alpha)$.

Recalling the moment map 
\bd
\mu: \mathbb{M}(\mathbf{v},\mathbf{w})\rightarrow \on{Lie}(\mathbb{G})^*
\ed
above,  the Nakajima quiver variety associated to dimension vector $\mathbf{v}$ and framing vector $\mathbf{w}$ is 
\bd
\Ms = \Ms_\theta(\mathbf{v},\mathbf{w}) := \mu\inv(0)/\!\!/_{\chi_0}\mathbb{G}_0.
\ed
We suppress the subscript $\theta$ in the rest of the paper.  Crawley-Boevey \cite[p.~261]{CB} shows that
$\mathbb{M}(\mathbf{v},\mathbf{w})\cong T^*\on{Rep}(Q,\alpha)$, intertwining the $\mathbb{G}_0$ and $\mathbb{G}$-actions, and identifying $\chi_0$-(semi)stability with $\chi$-(semi)stability.  Thus we may take the Hamiltonian reduction of $T^*\on{Rep}(Q,\alpha)$ with respect to $G(\alpha)$, using the stability condition determined by $\chi$ or equivalently $\theta$, and obtain $\Ms$ as the GIT quotient.

\subsection{Tautological Bundles and Nakajima's Section}\label{tautological bundles}
We continue with a quiver $Q = (I,\Omega)$.
Let $V^1$ and $V^2$ be $I$-graded vector spaces of dimension $\mathbf v$, and $W$ an $I$-graded vector space of dimension $\mathbf w$.  
One defines functors from $\mathbb{G}$-representations, respectively $\mathbb{G}\times\mathbb{G}$-representations, to $\mathbb{G}$-equivariant vector bundles on a $\mathbb{G}$-variety $Z$, respectively to $\mathbb{G}\times\mathbb{G}$-equivariant vector bundles on a $\mathbb{G}\times\mathbb{G}$-variety, by $R\mapsto \mathscr{R} := \theo\otimes_{\mathbb{C}} R$.   

In particular, each $V^j_i$ defines a $\mathbb{G}$-equivariant vector bundle $\mathscr{V}^j_i$ on $\mathbb{M}$, and the $\mathbb{G}\times\mathbb{G}$-representations $L(V^1, V^2), E(V^1, V^2), L(W,V^2), L(V^1, W)$ define $\mathbb{G}\times\mathbb{G}$-equivariant vector bundles 
\bd
\mathscr{L}(V^1, V^2),\hspace{1em} \mathscr{E}(V^1, V^2), \hspace{1em} \mathscr{L}(W,V^2), \hspace{1em} \mathscr{L}(V^1, W)
\ed
 on $\mathbb{M}\times\mathbb{M}$ (where $\mathbb{G}\times\mathbb{G}$ acts on $V^1$ via the first factor and on $V^2$ via the second factor).
\begin{remark}\label{remark:stacks}
In the language of stacks, these bundles are pullbacks along $\mathbb{M}/\mathbb{G}\times\mathbb{M}/\mathbb{G}\rightarrow B\mathbb{G}\times B\mathbb{G}$.  
\end{remark}

Nakajima defines $\mathbb{G}\times\mathbb{G}$-equivariant homomorphisms,
 \begin{equation}\label{vb-maps}
\xymatrix{
\mathscr{L}(V^1,V^2) \ar[r]^{\hspace{-6em}\sigma} & \mathscr{E}(V^1,V^2)\oplus \mathscr{L}(W,V^2)\oplus \mathscr{L}(V^1,W) \ar[r]^{\hspace{6em}\tau} & \mathscr{L}(V^1,V^2),
}
\end{equation}
where at a point $([B,i,j],[B',i',j'])\in\mathbb{M}\times\mathbb{M}$ the maps $\sigma$, $\tau$ are given by 
\begin{equation}\label{sigma-tau}
\sigma(\xi) = (B'\xi - \xi B,-\xi  i, j' \xi),
\hspace{2em}
\tau(C,a,b) = \epsilon B'C+\epsilon CB + i'b + aj.
\end{equation}

\begin{prop}\label{prop:Nakajima-maps}
Suppose $[B,i,j], [B',i',j']\in \mathbb{M}$.  
\mbox{}
\begin{enumerate}
\item If $[B',i',j']\in\mathbb{M}^{\on{s}}$ then $\sigma$ is injective in the fiber over $\big([B,i,j], [B',i',j']\big)$.
\item If $[B,i,j]\in\mathbb{M}^{\on{s}}$ then $\tau$ is surjective in the fiber over $\big([B,i,j], [B',i',j']\big)$.
\item $\tau\circ\sigma = 0$  in the fiber over $\big([B,i,j], [B',i',j']\big)\in\mu\inv(0)\times\mu\inv(0)$.
\end{enumerate}
\end{prop}

Now define a section $s$ of  $\mathscr{E}(V^1,V^2)\oplus \mathscr{L}(W,V^2)\oplus \mathscr{L}(V^1,W)$ by 
\begin{equation}\label{section-s}
s([B,i,j],[B',i',j']) = (0,-i',j).
\end{equation}
\begin{prop}\label{prop:Nakajima-section}
\mbox{}
\begin{enumerate}
\item
Over $\mathbb{M}\times\mathbb{M}$, we have $\tau(s)= 0$.
\item Viewing $s|_{\mu\inv(0)\times \mu\inv(0)^{\on{s}}}$ as a section of $\on{coker}(\sigma)$,
its vanishing locus $Z(s)$ in $\mu\inv(0)^{\on{s}}\times \mu\inv(0)^{\on{s}}$  is smooth and equals the locus of pairs $\big([B,i,j], [B',i',j']\big)$ for which 
\bd
\mathbb{G}\cdot[B,i,j] = \mathbb{G}\cdot[B',i',j'].
\ed
\end{enumerate}
\end{prop}
For the proofs of these propositions when the character is the one used in \cite{Nakajima Duke}, see \cite[p. 537 and Lemma 5.2]{Nakajima Duke}. We reprove the assertions in general in Proposition \ref{prop:complex-features} and Theorem \ref{resolution-prop}.

We now want to translate  the above in terms of the Crawley-Boevey quiver $Q^{\on{CB}}$.  Consider framed representations $[B,i,j], [B',i',j'] \in \mu\inv(0)^{\on{s}}\times\mu\inv(0)^{\on{s}}$, acting on the vector spaces $(V^1, W)$ and $(V^2, W)$ (both with associated dimensions $\mathbf{v}, \mathbf{w}$).  
We write 
$B^{\on{CB}}$, $(B')^{\on{CB}}$ for the associated representations of the preprojective algebra $\Pi^0(Q^{\on{CB}})$, and
$(V^\ell)^{\on{CB}}$ for their underlying vector spaces. Thus, one has 
\bd
(V^\ell)^{\on{CB}}_j = \begin{cases} V^\ell_i & \text{if $j=i\in I$};\\
\mathbb{C} & \text{if $j=\infty$}.\end{cases}
\ed 
Now
\begin{align}\label{CB1}
L\big((V^1)^{\on{CB}}, (V^2)^{\on{CB}}\big)  & = L(V^1, V^2) \oplus \Hom(\mathbb{C},\mathbb{C}),\\
\label{CB2}
E((V^1)^{\on{CB}}, (V^2)^{\on{CB}})  & = E(V^1, V^2)\oplus L(W, V^2)\oplus L(V^1, W).
\end{align}
The following is immediate from \eqref{CB1}, \eqref{CB2}, and Proposition \ref{prop:Nakajima-section}:
\begin{prop}\label{identification of partialzero}
\mbox{}
\begin{enumerate}
\item Under the identifications of \eqref{CB1}, \eqref{CB2}, the map 
\bd
L\big((V^1)^{\on{CB}}, (V^2)^{\on{CB}}\big) = L(V^1,V^2) \oplus\mathbb{C} \xrightarrow{\sigma \oplus s} E(V^1,V^2)\oplus L(W,V^2)\oplus L(V^1,W) = E((V^1)^{\on{CB}}, (V^2)^{\on{CB}})
\ed
is identified with the map
\bd
\partial_0: L\big((V^1)^{\on{CB}}, (V^2)^{\on{CB}}\big)\longrightarrow E((V^1)^{\on{CB}}, (V^2)^{\on{CB}})
\ed
defined by $\partial_0(\phi) = (B')^{\on{CB}}\phi - \phi B^{\on{CB}}$.  
\item Thus, for the dual map 
\bd
\partial_0^\vee: \mathscr{E}((V^1)^{\on{CB}}, (V^2)^{\on{CB}})  \longrightarrow  \mathscr{L}\big((V^1)^{\on{CB}}, (V^2)^{\on{CB}}\big)
\ed
 we have that $\on{coker}(\partial_0^\vee)$ is the direct image to $\mu\inv(0)^{\on{s}}\times\mu\inv(0)^{\on{s}}$ of a line bundle on the smooth subvariety of part (2) of Proposition \ref{prop:Nakajima-section}.
 \end{enumerate}
\end{prop}

\section{Graded Tripled Quivers and Their Moduli Spaces}
The present section is intended to provide a compactification of the moduli space of representations of the preprojective algebra $\Pi^0(Q)$ associated to a quiver $Q$.  For applications to Nakajima quiver varieties associated to a quiver $Q_0$, set $Q = Q_0^{\on{CB}}$, the Crawley-Boevey quiver associated to $Q_0$.
\subsection{Graded Tripling of a Quiver}\label{gradedtriple}
Let $(I,E)$ be a graph, $\alpha\in\mathbb{Z}^I_{\geq 0}$ a dimension vector for $I$.  Fix an orientation $\Omega$ defining a quiver $Q = (I,\Omega)$ as in Section \ref{sec:basicsofquivers}.
Fixing a closed interval $[a,b]\subset\mathbb{Z}$, we define a new quiver associated to $(I,\Omega)$, the {\em graded-tripled quiver}, denoted $\Qgtr$, as follows.  We give $\Qgtr$ the vertex set $I\times [a,b]$ where $I$ is the vertex set of $Q$.  If $E$ is the edge set of $Q$ and $H$ the associated set of pairs of an edge together with an orientation, we give $\Qgtr$ the arrow set
\bd
\big(H\times [a,b-1]\big) \cup \big(I\times [a,b-1]).   
\ed
Thus:
\begin{enumerate}
\item for each $h\in H$, $n\in [a,b-1]$  we have arrows $(h,n)$ with 
\bd
s(h,n) = (s(h), n) \hspace{1em} \text{and} \hspace{1em} t(h,n) = (t(h), n+1);
\ed
\item for each $i\in I$, $n\in [a,b-1]$ we have arrows $(i,n)$  with 
\bd
s(i,n) = (i,n) \hspace{1em} \text{and} \hspace{1em} t(i,n) = (i,n+1).
\ed
\end{enumerate}
  For example, taking $[a,b] = [0,1]$:

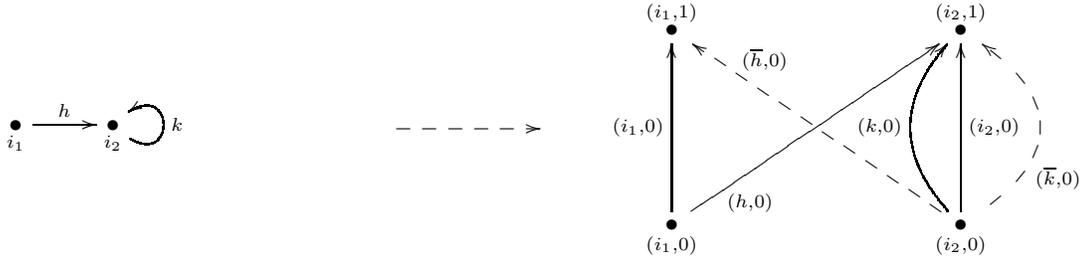
\def\tabularxcolumn#1{m{#1}}
\begin{tabularx}{\textwidth}{XXX}
$\ds\xymatrix{ \underset{i_1}{\bullet} \ar@[blue][r]^{h} &\underset{i_2}{\bullet}\ar@[red]@(dr,ur)[]_{k}}$
&
$\ds\xymatrix{ \mbox{}\ar@{-->}[rr] & & \mbox{}}$ 
& 
$\ds\mbox{} \hspace{-6em}\xymatrix{ \overset{(i_1,1)}{\bullet} & & &  \overset{(i_2,1)}{\bullet} \\
& & \\
\underset{(i_1,0)}{\bullet}\ar@[blue][uurrr]_(.2){(h,0)} \ar@[green][uu]^{(i_1,0)} & &  &\underset{(i_2,0)}{\bullet} \ar@[green][uu]_{(i_2,0)} \ar@[red]@(ul,dl)[uu]^{(k,0)}\ar@[red]@{-->}@/_7ex/[uu]_(.3){(\overline{k},0)}\ar@[blue]@{-->}[uulll]_(.75){(\overline{h},0)}
}
$
\end{tabularx}

\begin{remark}
Letting $b\rightarrow\infty$, the constructions extend {\em mutatis mutandis} to the case $[a,\infty)\subset \mathbb{Z}$.
\end{remark}

Given a dimension vector $\alpha$ for $Q$, we define a ``constant dimension vector'' $\algtr$ for $\Qgtr$ by
$\algtr_{i,n} = \alpha_i$ for all $i\in I, n\in [a,b]$.  
\subsection{Relations and Representations}\label{varieties-of-modules}
We will consider $\Qgtr$ as a quiver with relations.  Many of the relations are derived from those for the preprojective algebra $\Pi^0(Q)$.

We fix a decomposition $H = \Omega\sqcup\overline{\Omega}$ as in Section \ref{sec:basicsofquivers}, determining a function $\epsilon$.
\begin{notation}
We write:
\begin{enumerate}
\item $a_{h,n}$  for the generators of $k\Qgtr$ corresponding to arrows $(h,n)$ (where $h\in H, n\in [a,b-1]$);
\item $\pv_{i,n}$ for the generators of $k\Qgtr$ corresponding to arrows $(i,n)$ (where  $i\in I, n\in[a,-1]$).
\end{enumerate}
\end{notation}
\begin{defn}
We write $A := k\Qgtr/I$, where $I$ is the two-sided ideal in the path algebra $k\Qgtr$ generated by the following relations:
\begin{enumerate}
\item $\ds\sum_{h\in H} \epsilon(h)a_{\overline{h},n+1}a_{h,n}$, $n\in [a,b-2]$ (``preprojective relations'').
\item $\pv_{t(h),n+1}a_{h,n} - a_{h,n+1}\pv_{s(h),n}$ for all $n\in [a,b-2]$, $h\in H$.
\end{enumerate}
We note that it is immediate from condition (2) that the elements $\ds\pv_{n} := \sum_{i\in I}\pv_{i,n}$ are actually central in $A$: all other required relations hold trivially in the path algebra of $\Qgtr$.
\end{defn}

We write $\on{Rep}(\Qgtr,\algtr)$ for the space of representations of $\Qgtr$ with dimension vector $\algtr$: thus, fixing an $I\times [a,b]$-graded vector space $\ds V_{\bullet,\bullet} = \bigoplus_{i\in I, n\in [a,b]} V_{i,n}$ with dimension vector $\algtr$, we set
\bd
\on{Rep}(\Qgtr,V_{\bullet,\bullet}) = \left(\bigoplus_{h\in H, n\in [a,b-1]} \Hom(V_{s(h),n}, V_{t(h),n+1})\right) \bigoplus \left(\bigoplus_{i\in I, n\in[a,b-1]}\Hom(V_{i,n}, V_{i,n+1})\right).
\ed
We write $\on{Rep}(\Qgtr,\algtr)$ when $V_{i,n} = \mathbb{C}^{\algtr_{i,n}}$.  We also write
 \bd
\on{Rep}(A,V_{\bullet,\bullet})\subseteq\on{Rep}(\Qgtr,V_{\bullet,\bullet}), \hspace{1em}\text{respectively}\hspace{1em} \on{Rep}(A,\algtr)\subseteq\on{Rep}(\Qgtr,\algtr)
 \ed
  for the closed affine subscheme of representations of $A$ (that is, representations of $\Qgtr$ satisfying the relations generating $I$).  We will write
 $\ds\mathbb{G} = \prod_{i\in I} GL(\alpha_i)$
  for the group associated to $Q$ and dimension vector $\alpha$; then
$\Ggtr \cong \mathbb{G}\times [a,b]$
 naturally acts on the affine schemes $\on{Rep}(A,\algtr)\subseteq\on{Rep}(\Qgtr,\algtr)$. 
 \begin{remark}
 We note that this choice of notation is not entirely consistent with our earlier notation in the context of Nakajima quiver varieties.  When $Q = Q_0^{\on{CB}}$ is the Crawley-Boevey quiver associated to $Q_0$, we will write $\ds\mathbb{G}_0 = \prod_{i\in I_0}GL(\alpha_i)$.
 \end{remark}

Consider $\Pi^0 = \Pi^0(Q)$ as a graded algebra (with all generators corresponding to arrows $h\in H$ in degree $1$).  Let $\Pi^0[\pv]$ be the graded polynomial extension with $\deg(\pv)=1$.  
\begin{lemma}\label{A-computation}
Suppose $V_{\bullet,\bullet}$ is an $I\times [a,b]$-graded vector space.    
Letting $h\in\Pi^0$ act via $\ds\sum_n a_{h,n}\in \on{Rep}(\Qgtr,V_{\bullet,\bullet})$ and $\pv$ act via $\ds\sum_{i,n} \pv_{i,n}\in \on{Rep}(\Qgtr,V_{\bullet,\bullet})$,  the space of graded $\Pi^0[\pv]$-module structures on $V_{\bullet,\bullet}$ is naturally identified with  $\on{Rep}(A,V_{\bullet,\bullet})$. 
\end{lemma}

\subsection{From $\Pi^0$-Modules to $\Qgtr$-Representations}
Suppose we have a finite-dimensional representation $V = (V_i)_{i \in I}$ of the preprojective algebra $\Pi^0$ of dimension vector $\alpha$.  
\begin{construction}
We obtain a representation  of $A$ on a vector space $V_{\bullet, \bullet}$ of dimension vector $\algtr$ defined by:
\begin{enumerate}
\item setting $V_{i,n} := V_i$ for all $i\in [a,b]$;
\item  defining each $\pv_{i,n}: V_{i,n} = V_i \xrightarrow{\on{id}} V_i = V_{i,n+1}$ to act by shift of $\mathbb{Z}$-grading; and 
\item defining each generator of $A$ corresponding to $h\in H$ to act via $\Pi^0$ followed by grading shift.  
\end{enumerate}
\end{construction}
The construction determines a morphism of algebraic varieties (``induction'')
\bd
\mathsf{Ind}^\circ: \on{Rep}(\Pi^0,V)\longrightarrow \on{Rep}(A, V_{\bullet,\bullet}).
\ed

\bd
\text{Write } \hspace{2em} 
 \mathbb{G} = \prod_i GL(V_{i})\hspace{1em}\text{and}\hspace{1em}  \Ggtr = \prod_{(i,n)\in I\times[a,b]} GL(V_{i,n}) \cong \prod_{n\in [a,b]}\mathbb{G} \hspace{1em} \text{as above,}
\ed with the
diagonal homomorphism $\ds \on{diag}: \mathbb{G}\rightarrow \Ggtr \cong \prod_{n\in [a,b]} \mathbb{G}$.  Then the morphism $\mathsf{Ind}^\circ$ is $(\mathbb{G}, \Ggtr)$-equivariant.  We thus get a natural $\Ggtr$-equivariant morphism
\begin{equation}\label{Ind-map}
\mathsf{Ind}: \Ggtr\times_{\mathbb{G}} \on{Rep}(\Pi^0,V)\longrightarrow \on{Rep}(A, V_{\bullet,\bullet}).
\end{equation}
Thus, given a representation $(a_h: V_{s(h)}\rightarrow V_{t(h)})_{h \in H}$ of $\Pi^0$ on $V$, and $(g_{i,n})\in \Ggtr$, we have
\bd
\mathsf{Ind}\big((g_{i,n}), a_h\big) = (a_{h,n}, \pv_{i,n}) \, \text{where} \,  a_{h,n} = g_{t(h),n+1}a_h g_{s(h),n}\inv \, \text{and} \, 
\pv_{i,n} = g_{i,n+1} g_{i,n}\inv.
\ed

\begin{prop}
The map $\mathsf{Ind}$ of \eqref{Ind-map} defines an open immersion of $\Ggtr\times_{\mathbb{G}} \on{Rep}(\Pi^0,V)$ in $\on{Rep}(A,V_{\bullet,\bullet})$, whose image consists of those $(a_{h,n}, \pv_{i,n})$ for which:
\bd
(\dagger)  \hspace{10em} \text{$\pv_{i,n}$ is an isomorphism for all $n\in [a,b-1]$.} \hspace{12em} \mbox{}
\ed
\end{prop}
\begin{proof}
The condition $(\dagger)$ is clearly an open condition.  Given $(a_{h,n}, \pv_{i,n})$ satisfying $(\dagger)$, define
\bd
a_h := \pv_{t(h),a}\inv a_{h,a},  g_{i,a} := \on{Id}_i, g_{i,n} := \pv_{i,n-1}\pv_{i,n-2}\dots \pv_{i,a} \, \text{for $n\geq a+1$}.
\ed
Inductively applying the identity $e_{t(h), a+1} a_{h,a} = a_{h,a+1}e_{s(h),a}$, one calculates that 
$(g_{i,n})\cdot \mathsf{Ind}^\circ(a_h) = (a_{h,n}, \pv_{i,n})$.  This construction
$(a_{h,n}, \pv_{i,n})\mapsto \big((g_{i,n}), a_h)\in \Ggtr\times_{\mathbb{G}} \on{Rep}(\Pi^0,V)$ 
 is evidently inverse to $\mathsf{Ind}$ on the locus of those $(a_{h,n}, \pv_{i,n})$ that satisfy the condition $(\dagger)$.
\end{proof}
\begin{corollary}
The morphism of quotient stacks
\bd
 \mathsf{Ind}^\circ:  \on{Rep}(\Pi^0,V)/\mathbb{G}\longrightarrow \on{Rep}(A,V_{\bullet,\bullet})/\Ggtr
\ed
is an open immersion.
\end{corollary}

\begin{remark}\label{extension of V}
We note that if $V_{\bullet,\bullet}$ lies in the open image of $\mathsf{Ind}$, then it uniquely determines an $I\times\mathbb{Z}$-graded $\Pi^0[\pv]$-module $\wt{V}_{\bullet,\bullet}$ with $\wt{V}_{i,n} = \alpha_i$ for {\em all} $n\in\mathbb{Z}$ and $i\in I$.  In other words, $V_{\bullet,\bullet}$ uniquely extends ``upwards and downwards'' to all graded degrees compatibly with the $\Pi^0[\pv]$-action.
\end{remark}

\subsection{Stability for Crawley-Boevey Quivers}\label{stability for CB}
Suppose that $Q_0 = (I_0,\Omega_0)$ is a quiver with dimension vector $\mathbf{v}$ and framing vector $\mathbf{w}$ and that $Q = Q_0^{\on{CB}} = (I,\Omega)$ is the associated Crawley-Boevey quiver, with dimension vector $\alpha$ so that $\alpha_\infty = 1$ and $\alpha|_{I_0} = \mathbf{v}$.  We fix $[a,b] \subset \mathbb{Z}$ and let $\Qgtr$ denote the quiver constructed above from $Q$.  We write
$\algtr$ for the associated dimension vector: thus,
\bd
\algtr_{i,n} = \alpha_i = \begin{cases} \mathbf{v}_i & i\in I_0\\ 1 & i=\infty.\end{cases}
\ed
\begin{center}
{\bf Assume given a nondegenerate stability condition $\theta = \sum \theta_i \delta_i$ for $Q$ with respect to $\alpha$.}
\end{center}
\begin{remark}
In particular, we have $\theta_\infty \neq 0$. 
\end{remark}  

We want to choose a stability condition $\thetagtr$ for $\Qgtr$ with the following properties:
\begin{enumerate}
\item $\thetagtr$ is nondegenerate with respect to $\algtr$.  In particular, the semistable and stable points of $\on{Rep}(\Qgtr,\algtr)$ coincide.
\item If $V$ is a representation associated to a representation of the preprojective algebra $\Pi^0(Q)$, then $V$ is $\thetagtr$-stable if and only if the corresponding $\Pi^0(Q)$-representation is $\theta$-stable.
\end{enumerate}
We first remind the reader that $\delta_{i,n}(M) := \dim(M_{i,n})$; we will write $\theta$ as a linear combination of the $\delta_{i,n}$.  Also, we note that it suffices to construct a {\em rational} linear functional $\thetagtr$, since any positive integer multiple of $\thetagtr$ evidently defines the same stable and semistable loci.  

In our construction of $\thetagtr$, we will want to fix a positive integer 
\begin{equation}\label{bounds}
\ds T \gg 0.
\end{equation}
   We fix an ordering on the vertices of $Q_0$, identifying
$I = \{1, \dots, r\}$.  We write $\thetagtr$ as a sum of terms:
\bd
\theta^{\on{lg}} = T^{r+1}\big[\delta_{\infty,b}-\delta_{\infty,a}\big] + \sum_{i=1}^r T^i\big[\delta_{i,b}-\delta_{i,a}\big],
\hspace{4em}
\theta^{\on{mid}} = \sum_{i\in I} \theta_i \delta_{i,a},
\ed
\bd 
\text{and} \hspace{4em} 
\theta^{\on{sm}} = -\sum_{i=1}^r T^{-i}\delta_{i,a} + T^{-r-1} \sum_{(i,n)\in I\times (a,b)} \delta_{i,n}.
\ed
Finally, we write $C := \theta^{\on{lg}}(\algtr) + \theta^{\on{mid}}(\algtr) + \theta^{\on{sm}}(\algtr)$ and write
\bd
\thetagtr := \theta^{\on{lg}} + \theta^{\on{mid}} + \theta^{\on{sm}} - C\delta_{\infty,a}.  
\ed
We note that $\theta^{\on{lg}}(\algtr) =0$, so $C$ is bounded independent of $T$.  
Also, since $\delta_{\infty, a}(\algtr) = 1$, we get $\thetagtr(\algtr)=0$.  
\begin{lemma}\label{semistable equals stable}
For fixed dimension vector $\alpha$ (and thus $\algtr$) and choices as in \eqref{bounds}, 
\bd
\thetagtr(M)\neq 0 \hspace{1em}\text{for} \hspace{1em}   0\not\subseteq M\not\subseteq V.
\ed
\end{lemma}
\begin{proof}
Assume that $\thetagtr(M) = 0$.  Write $m_{i,n}:= \delta_{i,n}(M) = \dim(M_{i,n})$.  Since the coefficients of $T$ in $\thetagtr$ are bounded independent of $T$, 
we conclude that each $T$-coefficient of $\thetagtr(M)$ must vanish.  In particular, $\theta^{\on{lg}}(M) = 0$ and thus $m_{i,b}=  m_{i,a}$ for all $i\in I$.

Since $m_{\infty,a} \in \{0,1\}$, we consider the two cases:

\vspace{.4em}

\noindent
{\em Case 1.}\, $m_{\infty, a} = m_{\infty,b} =  0$.  \; \;  In this case $0 = \thetagtr(M) = \theta^{\on{mid}}(M) +\theta^{\on{sm}}(M)$, and again for $T\gg 0$ each coefficient of $T$ must vanish.  From $\theta^{\on{sm}}(M) = 0$ we get
$m_{i,a}= 0$ for all $i\in I_0$, and $\ds\sum_{(i,n)\in I\times (a,b)} m_{i,n} =0$ implying $m_{i,n} = 0$ for $(i,n)\in I\times(a,b)$.  Combined with the equality $m_{i,b}=  m_{i,a}$ for all $i\in I$ from above, we conclude $M=0$.

\vspace{.4em}

\noindent
{\em Case 2.}\, $m_{\infty, a} = m_{\infty,b} = 1$.  \;\;  Then
\bd
0  =\thetagtr(M) = \theta^{\on{mid}}(M) + \theta^{\on{sm}}(M) - \big[\theta^{\on{mid}}(\algtr) +\theta^{\on{sm}}(\algtr)].
\ed
Again, considering term-by-term in powers of $T$, we find that $m_{i,a} = \alpha_{i,a}$ for $i\in I_0$; and then 
$\ds\sum_{(i,n)\in I\times (a,b)} m_{i,n} = \ds\sum_{(i,n)\in I\times (a,b)} \algtr_{i,n}$ implying (since $m_{i,n}\leq \algtr_{i,n}$) that $m_{i,n} = \algtr_{i,n}$ for $(i,n)\in I\times (a,b)$.  Combined with the equality $m_{i,b}=  m_{i,a}$ for all $i\in I$, we conclude that  $m_{i,n} = \algtr_{i,n}$ for all $(i,n)\in I\times [a,b]$, i.e. $M=V$.
\end{proof}

\begin{prop}\label{prop:quiver-moduli}
With respect to $\thetagtr$ as above, we have:
\mbox{}
\begin{enumerate}
\item The semistable and stable loci of $\on{Rep}(\Qgtr,\algtr)$ coincide, as do those of $\on{Rep}(A,\algtr)$.
\item Every stable point of $\on{Rep}(\Qgtr,\algtr)$ is generated as an $A$-module in degree $a$.
\item If $V_{\bullet,\bullet}$, $W_{\bullet,\bullet}$ are vector spaces with dimension vector $\algtr$, equipped with $A$-module structures making them stable, then $\Hom_{A}(V_{\bullet,\bullet}, W_{\bullet,\bullet})$ is $1$-dimensional if $V_{\bullet,\bullet}$ and $W_{\bullet,\bullet}$ are isomorphic as $A$-modules and is $0$-dimensional otherwise.  
\item For a representation $V$ of $\Pi^0$ of dimension vector $\alpha$, $V$ is stable with respect to $\theta$ if and only if $\mathsf{Ind}^\circ(V) \in \on{Rep}(A,\algtr)$ is stable with respect to $\thetagtr$.
\end{enumerate}
\end{prop}
\begin{proof}
(1) This is the content of Lemma \ref{semistable equals stable}.

(2)  Supposing $V$ is stable, let $M$ be the subrepresentation generated by $V_{I\times\{a\}}$.  Arguing as in the proof of Lemma \ref{semistable equals stable},  we have that $\theta^{\on{lg}}(M) < 0$, and therefore $V$ is unstable,  unless $m_{i,b} = m_{i,a}$ for all $i\in I$.  We conclude that $m_{i,b} = m_{i,a}$ for all $i\in I$ and hence that
\bd
\thetagtr(M) = \theta^{\on{mid}}(M) + \theta^{\on{sm}}(M) - \big[\theta^{\on{mid}}(\algtr) +\theta^{\on{sm}}(\algtr)].
\ed 
Noting that $\theta^{\on{mid}}(M) = \theta^{\on{mid}}(V)$ by definition
and analyzing $\theta^{\on{sm}}(M) -\theta^{\on{sm}}(V)$ term-by-term in powers of $T$,
we find that $\thetagtr(M)<0$ unless $m_{i,n} = \algtr_{i,n}$ for all $(i,n)\in I\times (a,b)$, and thus stability of $V$ implies $M=V$.  

(3)  is standard.
%%Suppose $\rho: V_{\bullet,\bullet}\rightarrow W_{\bullet, \bullet}$ is a $k\Qgtr$-linear map.  By stability of $V_{\bullet,\bullet}$,  if $\theta$ is nonzero and not an isomorphism then $\theta\big(\on{ker}(\rho)\big) >0$; by linearity of $\theta$, it follows that $\theta\big(\on{coker}(\rho)\big) <0$, contradicting stability of $W_{\bullet, \bullet}$.  Thus any nonzero homomorphism $\rho$ is an isomorphism.  We may thus assume $W_{\bullet, \bullet} = V_{\bullet, \bullet}$.  If $\rho\in\End_A(V_{\bullet,\bullet})$ is nonzero, hence invertible, let $\lambda$ be any nonzero eigenvalue of $\rho$; then $\rho - \lambda\cdot\on{Id}$ is non-invertible, hence is zero.  

(4)  Consider a representation $V$ of $\Pi^0(Q)$.  As before, we write $\algtr$ for the dimension vector of $\mathsf{Ind}^\circ(V)$ where $V$ has dimension vector $\alpha$.   For any sub-representation 
$M\subseteq \mathsf{Ind}^\circ(V)$, write $m_{i,n} = \dim(M_{i,n})$.

Because $\pv_{i,n}: \mathsf{Ind}^\circ(V)_{i,n} \rightarrow \mathsf{Ind}^\circ(V)_{i,n+1}$ is an isomorphism for each $n\in [a,b-1]$, we have, for any sub-representation $M$, that  $m_{i,n+1}\geq \dim m_{i,n}$ for all $(i, n)\in I\times [a,b-1]$.  Analyzing $\thetagtr(M)$ term-by-term in powers of $T$, we conclude that $\thetagtr(M)>0$, and thus $M$ is irrelevant to the stability of $\mathsf{Ind}^\circ(V)$,
 unless $m_{i,b} = m_{i,a}$ for all $i\in I$, i.e., unless $M = \mathsf{Ind}^\circ(V')$ for some $\Pi^0(Q)$-submodule $V'\subseteq V$.  
 
 Thus, suppose $M = \mathsf{Ind}^\circ(V')$ for some $\Pi^0(Q)$-submodule $V'\subseteq V$.  Write $\alpha'$ for the dimension vector of $V'$.  Then
$ \thetagtr(M) = \theta(\alpha') + \theta^{\on{sm}}(M) -\big[\theta(\alpha) + \theta^{\on{sm}}(\algtr)]\alpha'_\infty.$

\vspace{.4em}

\noindent
{\em Case 1.}\,
$\alpha'_\infty = 0$.
 \; \;
In this case, $\thetagtr(M) = \theta(\alpha') + \theta^{\on{sm}}(M)$.  If $\theta(\alpha')<0$, so $V'$ destablizes $V$, then we see that $\thetagtr(M)<0$, so $M$ destabilizes $\mathsf{Ind}^\circ(V)$.  On the other hand if $\theta(\alpha')>0$ then $\thetagtr(M)>0$ as well.  Thus in this case, $V'$ destabilizes $V$ if and only if 
$\mathsf{Ind}^\circ(V')$ destabilizes $\mathsf{Ind}^\circ(V)$.

\vspace{.4em}

\noindent
{\em Case 2.}\,
$\alpha'_\infty = 1$.
 \; \;
Then, as in Case 2 of Lemma \ref{semistable equals stable},
\bd\thetagtr(M) = \theta^{\on{mid}}(M) + \theta^{\on{sm}}(M) - \big[\theta^{\on{mid}}(\algtr) +\theta^{\on{sm}}(\algtr)].
\ed
The leading term in $T$ is $\theta^{\on{mid}}(M)-\theta^{\on{mid}}(\algtr) = \theta(\alpha')$.  Thus $\thetagtr(M)<0$ if and only if $\theta(\alpha')<0$, and so 
$V'$ destabilizes $V$ if and only if 
$\mathsf{Ind}^\circ(V')$ destabilizes $\mathsf{Ind}^\circ(V)$.
  This completes the proof.
\end{proof}

As in \cite[Proposition 4.3]{King}, since $\Qgtr$ has no oriented cycles we obtain a {\em projective} quotient
\bd
\overline{\Ms} := \on{Rep}(A,\algtr)/\!\!/_{\chigtr}\Ggtr.
\ed

\begin{corollary}
The natural map $\on{Ind}: \Ms\rightarrow \overline{\Ms}$ is an open immersion of the quiver variety $\Ms$ in a projective scheme.
\end{corollary}
\begin{remark}
Although it appears that $\overline{\Ms}$ is nonsingular and connected in the instances we care about, we do not need this.  Instead, we may replace $\overline{\Ms}$  by the closure of $\Ms$ in $\overline{\Ms}$ and give that closure the reduced scheme structure.  Thus, 
\begin{center}
{\em in what follows we always assume without comment that $\overline{\Ms}$ is integral and projective.}
\end{center}
\end{remark}

\section{A Perfect Complex on $\Ms\times\overline{\Ms}$}
We note that the construction in this section is similar to the one in Section 5 of \cite{Nakajima Duke}.  However, we wish to emphasize that Nakajima's framings are {\em not} explicitly present in this section: for applications to Nakajima quiver varieties with nonzero framing, one should take $Q = (Q_0)^{\on{CB}}$ to be the Crawley-Boevey quiver associated to the quiver $Q_0$ used in Nakajima's constructions.

Fix a quiver $Q$ and a dimension vector $\alpha$.  Let $V_{\bullet,\bullet}$, $W_{\bullet,\bullet}$ be two $I\times [a,b]$-graded vector spaces with dimension vector $\algtr$; we write $V^\ell_{i,n}$ for the $(i,n)$-graded piece.
\begin{remark}
We again emphasize that $V_{\bullet,\bullet}$, $W_{\bullet,\bullet}$ will be endowed with the structure of representations of $\Qgtr$ satisfying the relations of $A$.  Our choice of notation for the space $W_{\bullet,\bullet}$ is {\em not} meant to indicate any relationship to Nakajima's framing vector space $(W_i)_{i\in I}$.
\end{remark}

\begin{convention}
We now fix an $N\geq 2$ and set $[a,b]=[0,N]$ in the definitions of $\Qgtr$, $\algtr$, $A$.  
\end{convention}

Suppose that we choose representations of $A$ in $V_{\bullet,\bullet}, W_{\bullet,\bullet}$; we write $(a^V, \pv^V) = (a^V_{h,n}, \pv^V_{i,n})$, respectively 
$(a^W, \pv^W) = (a^W_{h,n}, \pv^W_{i,n})$ to denote these two structures.  We also write 
\bd
\ds a^V_n = \sum_{h\in H}a^V_{h,n} \hspace{1em}\text{and}\hspace{1em} \pv^V_n = \sum_{i\in I} \pv^V_{i,n}, \hspace{2em} \text{and similarly for $W$.  }
\ed

\begin{notation}
Given a linear operator $L$ between such graded vector spaces, we sometimes write $L_i$ to mean ``the component of the operator acting on the $i$-graded piece of the domain;'' for example,  the notation $[(\pv^V_1)\inv \lambda \pv^W_1]_{s(\overline{h}_0)}$ is used in Equation \eqref{precise} to mean the component of $(\pv^V_1)\inv \lambda \pv^W_1$ acting at vertex $s(\overline{h}_0)$.  We also remind the reader that $s(\overline{h}) = t(h)$, which explains some possibly confusing indices in the proof of Proposition \ref{prop:complex-features} below.
\end{notation}
\begin{assumption}\label{assumption-of-invertibility}
We assume that the representation $V_{\bullet,\bullet}$ lies in the image of $\on{Ind}$: in other words, the linear operators $\pv^V_{i,n}$ are invertible for $n\in[0,N-1]$.
\end{assumption}

Consider the vector spaces and maps, graded so  $E(V_{\bullet, 0}, W_{\bullet, 1})$ lies in cohomological degree $0$,
\begin{equation}\label{eq:perfect-complex}
L(V_{\bullet, 0},W_{\bullet, 0}) \xrightarrow{\partial_0} E(V_{\bullet, 0}, W_{\bullet, 1})\xrightarrow{\partial_1} L(V_{\bullet, 0}, W_{\bullet, 2}),
\end{equation}
defined as follows: given $\phi\in L(V_{\bullet, 0}, W_{\bullet, 0})$ and $\psi\in E(V_{\bullet, 0},W_{\bullet, 0})$,  we let
\begin{equation}\label{mapformulas}
\partial_0(\phi) = a^W_0\phi - \pv^W_0\circ\phi \circ(\pv^V_1)\inv a^V_0, \hspace{2em} \partial_1(\psi) = (\epsilon a^W_1)\psi + \pv^W_1\circ \psi \circ(\pv^V_1)\inv (\epsilon a^V_0).
\end{equation}

\begin{prop}\label{prop:complex-features}
\mbox{}
\begin{enumerate}
\item The kernel of $\partial_0$ is naturally identified with a subspace of $\Hom_{A}(V_{\bullet, \bullet}, W_{\bullet, \bullet})$.
\item The composite $\partial_1\circ\partial_0$ is zero.
\item If $[a,b]=[0,2]$, the cokernel of $\partial_1$ is naturally identified with  $\Hom_{A}(W_{\bullet, \bullet}, V_{\bullet, \bullet})^*$.
\end{enumerate}
\end{prop}
We note that for assertion (3), we use in a fundamental way that Remark \ref{extension of V} applies to $V_{\bullet, \bullet}$.
\begin{proof}
If $\partial_0(\phi) = 0$, then we may define a linear map $\Phi_\bullet: V_{\bullet, \bullet}\rightarrow W_{\bullet, \bullet}$ by
$\Phi_n = \pv^W_{n-1}\dots\pv^W_0\circ \phi \circ (\pv^V_{n-1}\dots \pv^V_0)\inv$. 
  It is immediate from the construction that $\Phi_\bullet$  commutes with the operators $\pv_n$ in the obvious sense.  
  Similarly, since $\partial_0(\phi)= 0$ we get that $a^W_0\Phi_0 = \Phi_1a^V_0$; it is immediate by induction
   that $\Phi_\bullet$ is compatible with all operators $a$ in the obvious sense.  Thus $\Phi_\bullet\in \Hom_{A}(V_{\bullet, \bullet}, W_{\bullet, \bullet})$.  
   Since $\pv\in A$ acts invertibly on $V_{\bullet, \bullet}$ in the appropriate range, any such $\Phi_\bullet$ is determined uniquely by $\Phi_0 = \phi$ by the above construction, proving assertion (1).
   
 For assertion (2), we calculate:
 \begin{multline}
 \partial_1\partial_0(\phi) = (\epsilon a^W_1) a^W_0\phi 
 - (\epsilon a^W_1)\pv^W_0\phi (\pv^V_1)\inv a^V_0
 + \\
 \pv^W_1
 a^W_0\phi
  (\pv^V_1)\inv (\epsilon a^V_0)
 -
 \pv^W_1
 \pv^W_0\phi (\pv^V_1)\inv a^V_0
  (\pv^V_1)\inv (\epsilon a^V_0).
  \end{multline}
Now 
\bd
 - (\epsilon a^W_1)\pv^W_0\phi (\pv^V_1)\inv a^V_0
 +
 \pv^W_1
 a^W_0\phi
  (\pv^V_1)\inv (\epsilon a^V_0)  
  = 
   - \pv^W_1 (\epsilon a^W_1)\phi (\pv^V_1)\inv a^V_0
 +
 \pv^W_1
 a^W_0\phi
  (\pv^V_1)\inv (\epsilon a^V_0) 
  = 0.
  \ed
  Thus to prove (2) it suffices to show that 
  \bd
  (\epsilon a^W_1) a^W_0\phi  -
 \pv^W_1
 \pv^W_0\phi (\pv^V_1)\inv a^V_0
  (\pv^V_1)\inv (\epsilon a^V_0) 
  = 
   (\epsilon a^W_1) a^W_0\phi  -
 \pv^W_1
 \pv^W_0\phi (\pv^V_1)\inv 
 (\pv^V_2)\inv a^V_1
   (\epsilon a^V_0)
  =0
  \ed
  However, $ (\epsilon a^W_1) a^W_0 = 0 = a^V_1
   (\epsilon a^V_0)$ is immediate from the preprojective relations.

  We now turn to assertion (3).  Suppose $\lambda: W^2\rightarrow V^0$ is an $I$-graded linear map.  We have that $\on{tr}\big(\lambda\partial_1(\psi)\big) = 0$ for all $\psi\in E(V^0,W^1)$ if and only if
  \bd
  0 = \on{tr}\big(\lambda (\epsilon a^W_1)\psi\big) + \on{tr}\big(\lambda \pv^W_1\circ \psi \circ(\pv^V_1)\inv (\epsilon a^V_0)\big) 
  = 
  \on{tr}\big(\big(\lambda (\epsilon a^W_1) + (\pv^V_1)\inv (\epsilon a^V_0)\lambda \pv^W_1\big)\circ
  \psi \big)
  \ed
for all $\psi$, if and only if 
\begin{equation}\label{psi condition}
\lambda (\epsilon a^W_1) + (\epsilon a^V_0) (\pv^V_1)\inv \lambda \pv^W_1 = 0.
\end{equation}
More precisely, this formula ``unpacks'' as follows.  Suppose that $\psi  = (\psi_h)_{h\in H}$ and assume given an $h_0\in H$ with $\psi_h = 0$ for $h\neq h_0$.  Then the trace condition reads
\bd
\on{tr}\left[\lambda_{t(h_0)} (\epsilon a^W_{\overline{h}_0,1})\psi_{h_0} + (\epsilon a^V_{\overline{h}_0,0}) [(\pv^V_{s(\overline{h}_0),1})\inv \lambda_{s(\overline{h}_0)} (\pv^W_{t(h_0),1})] \psi_{h_0}\right] = 0.
\ed
Since $\psi_{h_0}: V^0_{s(h_0)} \rightarrow W^1_{t(h_0)}$ is arbitrary, it follows that
\begin{equation}\label{precise}
\lambda_{t(h_0)} (\epsilon a^W_{\overline{h}_0,1}) + (\epsilon a^V_{\overline{h}_0,0}) [(\pv^V_{1})\inv \lambda \pv^W_{1}]_{s(\overline{h}_0)}= 0 = 
\lambda_{t(h_0)} (a^W_{\overline{h}_0,1}) + (a^V_{\overline{h}_0,0}) [(\pv^V_1)\inv \lambda \pv^W_1]_{s(\overline{h}_0)}.
\end{equation}

By the nondegeneracy of the trace pairing, we obtain:
\begin{lemma}\label{nondegeneracy lemma}
The cokernel of $\partial_1$ is naturally dual to the space of those $\lambda$ satisfying \eqref{psi condition}.
\end{lemma}

We now use  Assumption \ref{assumption-of-invertibility}
and Remark \ref{extension of V} to see that $V_{\bullet, \bullet}$ lifts to an $I\times\mathbb{Z}$-graded $\Pi^0[\pv]$-module $\wt{V}_{\bullet, \bullet}$ with $\dim(V_{i,n}) = \alpha_i$ for all $i \in I$ and $n\in\mathbb{Z}$, in such a way that $\wt{V}_{\bullet, \bullet +1} \cong \wt{V}_{\bullet,\bullet}$ via multiplication by $\pv$.  More precisely, writing $(\wt{V}_{\bullet, \bullet})_{[i,j]} := \wt{V}_{\bullet, \bullet\geq i}/\wt{V}_{\bullet, \bullet\geq j+1}$, we extend $\lambda$ to a graded linear map
$\Lambda_\bullet: W_{\bullet, \bullet}\rightarrow \wt{V}_{\bullet, \bullet}(-2)_{[0,2]}\cong V_{\bullet, \bullet}$ by taking 
\bd
\Lambda_2 = \lambda, \hspace{1em}
\Lambda_1 = - (\pv^V_1)\inv \lambda \pv^W_1,  \hspace{1em}
\Lambda_0 = (\pv^V_0)\inv  (\pv^V_1)\inv \lambda \pv^W_1 \pv^W_0,
\ed
similarly to our construction of $\Phi_\bullet$ above.  As in our construction of $\Phi_\bullet$, it follows from Equation \eqref{psi condition} that $\Lambda_\bullet$ is indeed a graded $A$-module homomorphism; and that any graded $A$-module homomorphism $\Lambda_\bullet: W_{\bullet, \bullet}\rightarrow V_{\bullet, \bullet}$ is uniquely determined by $\Lambda_2=\lambda$ by this construction, completing the proof.
\end{proof}

\begin{corollary}\label{cor:perfect-complex} 
When $Q$ is a Crawley-Boevey quiver and $[a,b]=[0,2]$, then the complex \eqref{eq:perfect-complex} descends to a perfect complex $C$ on $\Ms\times\overline{\Ms}$.
\end{corollary}
\begin{proof}
 When $Q = (Q_0)^{\on{CB}}$ is a Crawley-Boevey quiver, we have $\mathbb{G} \cong \mathbb{G}_0\times \Gm$, where $\ds\mathbb{G}_0 = \prod_{i\in I_0}GL(\mathbf{v}_i)$, $\Gm$ acts trivially on the stable locus, and $\mathbb{G}_0$ acts freely on the stable locus of $\on{Rep}(\Pi^0(Q),\alpha)$ with quotient $\Ms$. Similarly, $\Ggtr\cong (\mathbb{G}_0)^3\times\Gm^{3}$; the subgroup $(\Ggtr)_0 = (\mathbb{G}_0)^3\times\Gm^2 \times \{1\}$ acts freely on $\on{Rep}(A,\algtr)^{\on{s}}$ with quotient $\Ms$.  
Since the complex defined by \eqref{eq:perfect-complex} is $(\Ggtr)_0\times (\Ggtr)_0$-equivariant, it descends to a perfect complex $C$ on $\Ms\times\overline{\Ms}$.  
\end{proof}

\section{Proofs of Theorems \ref{main thm}, \ref{complex-oriented}, and \ref{derived cat}}
Let $Q_0$ be a quiver with dimension vector $\mathbf{v}$ and framing vector $\mathbf{w}$, and let $Q = Q_0^{\on{CB}}$ be the  Crawley-Boevey quiver associated to $Q_0$ and $\mathbf{w}$.  
\begin{center}
{\bf We take $[a,b] = [0,2]$ in the definitions of $\Qgtr$, etc. }
\end{center}
Let $\Ms\hookrightarrow\overline{\Ms}$ denote the compactification of the quiver variety constructed in Section \ref{stability for CB}.   We wish to modify slightly the complex of \eqref{eq:perfect-complex} and Corollary \ref{cor:perfect-complex}.  Thus, we consider the splitting
\bd
L(V_{\bullet, 0},W_{\bullet, 0})= 
L(V_{\bullet, 0},W_{\bullet, 0})_{I_0} \oplus {\mathbb C} := 
\left[ \bigoplus_{i\in I_0} \Hom(V_{i, 0},W_{i, 0})\right] \oplus \Hom(V_{\infty,0}, W_{\infty, 0}). 
\ed 
\bd
\text{We write} \hspace{2em}
\delta_0 = \partial_0|_{L(V_{\bullet, 0},W_{\bullet, 0})_{I_0}}: 
L(V_{\bullet, 0},W_{\bullet, 0})_{I_0} \rightarrow E(V_{\bullet,0}, W_{\bullet, 1}).
\ed

Similarly, we consider the splitting
\bd
L(V_{\bullet, 0},W_{\bullet, 2})= 
L(V_{\bullet, 0},W_{\bullet, 2})_{I_0} \oplus {\mathbb C} := 
\left[ \bigoplus_{i\in I_0} \Hom(V_{i, 0},W_{i, 2})\right] \oplus \Hom(V_{\infty,0}, W_{\infty,2}) 
\ed
and write $\delta_1 = \pi\circ\partial_1$ for the composite of $\partial_1$ followed by the projection
\bd
\pi: L(V_{\bullet, 0},W_{\bullet, 2}) \twoheadrightarrow L(V_{\bullet, 0},W_{\bullet, 2})_{I_0}.
\ed
It is immediate from Corollary \ref{cor:perfect-complex} that we obtain a complex on $\Ms\times\overline{\Ms}$, namely
\begin{equation}\label{real-complex}
R: \mathscr{L}(V_{\bullet, 0},W_{\bullet, 0})_{I_0} \xrightarrow{\delta_0} \mathscr{E}(V_{\bullet, 0}, W_{\bullet, 1})\xrightarrow{\delta_1} \mathscr{L}(V_{\bullet, 0}, W_{\bullet, 2})_{I_0}.
\end{equation}
\begin{remark}\label{real is R}
The complex \eqref{real-complex} is evidently of the form \eqref{R summary}.
\end{remark}

\begin{thm}\label{resolution-prop}
For the complex $R$ of \eqref{real-complex}, we have:
\begin{enumerate}
\item  $\delta_0$ is injective and $\delta_1$ is surjective on each fiber.  In particular, $\mathcal{H}^1(R) = 0 = \mathcal{H}^1(R^\vee)$, and $\mathcal{H}^0(R)$ is a vector bundle on $\Ms\times\overline{\Ms}$.
\item the map $\mathbb{C} = \Hom(V_{\infty,0}, W_{\infty,0})\rightarrow E(V_{\bullet,0},W_{\bullet, 1})$ defines a section $s$ of $\mathcal{H}^0(R)$ whose scheme-theoretic zero locus is the graph $\Gamma$ of the inclusion $\Ms\hookrightarrow \overline{\Ms}$.
\item $\on{rk}(R) = \dim(\overline{\Ms})$.
\end{enumerate}
\end{thm}
\begin{proof}
(1) By Proposition \ref{prop:complex-features}, when $V_{\bullet, \bullet}$ and $W_{\bullet,\bullet}$ are stable, $\on{ker}(\partial_0)$ is zero or consists of multiples of the identity endomorphism of $V_{\bullet,\bullet} \cong W_{\bullet,\bullet}$; in either case, we have $\on{ker}(\partial_0)\cap L(V_{\bullet, 0},W_{\bullet, 0})_{I_0} = 0$.  Thus $\delta_0$ is injective on each fiber.  

Similarly either $\on{coker}(\partial_1)$ is zero, or else $V_{\bullet,\bullet}\cong W_{\bullet,\bullet}$ and $\on{coker}(\partial_1)\cong \Hom(W_{\bullet,\bullet}, V_{\bullet,\bullet})^* \cong \mathbb{C}$ by stability of $V_{\bullet,\bullet}$ and $W_{\bullet,\bullet}$; in the latter case, since $\on{im}(\partial_1)$ has codimension $1$, its projection on $L(V_{\bullet, 0}, W_{\bullet, 2})_{I_0}$ must be surjective: otherwise $\on{im}(\partial_1)\cap \Hom(V_{\infty,0},W_{\infty,2}) \neq 0$, but (by stability) every nonzero element of its dual $\Hom(W_{\bullet,\bullet}, V_{\bullet,\bullet})$ is nonzero at the vertex $\infty$.  We conclude that $\delta_1$ is surjective on each fiber, concluding the proof of assertion (1).

(2) By Proposition \ref{prop:complex-features}, the cohomologies $H^1(C)$ and $H^1(C^\vee)$ are supported set-theoretically on the graph $\Gamma$ of the inclusion $\Ms\hookrightarrow\overline{\Ms}$.  
It follows that the set-theoretic zero locus of the section $s$ of assertion (2) is $\Gamma$.  
Thus, to prove the scheme-theoretic assertion, we may restrict $R$ to $\Ms\times\Ms$.

Supposing, then, that both $\pv^V_{i,n}$ and $\pv^W_{i,n}$ act invertibly for $n=0,1$, and applying appropriate automorphisms of $V_{\bullet,\bullet}$ and $W_{\bullet,\bullet}$, we may assume that $a^V = a^W$ and all $\pv^V_{i,n}$ and $\pv^W_{i,n}$ are identity matrices.  Let $\mathbb{C}[\hbar]$ denote the ring of dual numbers and let $a^V + \hbar b^V$, $a^W + \hbar b^W$ be first-order deformations of $V_{\bullet, \bullet}$, $W_{\bullet,\bullet}$.  It is immediate from the formulas \eqref{mapformulas} that the linearization of the map $\partial_0$ of \eqref{eq:perfect-complex} is given by
$\phi \mapsto b_0^W \phi - \phi b_0^V$.  If the linearization is of less than full rank, then by Proposition \ref{prop:complex-features}(1) there is a homomorphism $0\neq \phi \in \Hom(V_{\bullet,\bullet}, W_{\bullet,\bullet})$ with   $b^W \phi = \phi b^V$.  Then  
the map $\on{Id}+ \hbar\phi$ intertwines $a^V + \hbar b^V$ and $a^W + \hbar b^W$: in other words, the differential of $\partial_0$ is degenerate only in directions tangent to $\Gamma$, which implies the assertion about $s$.

(3) The rank assertion is immediate by direct calculation as in \cite{Nakajima Duke}.
\end{proof}

\begin{proof}[Proof of Theorem \ref{main thm}]
Let $d= \on{dim}(\Ms)$.  
By Theorem \ref{resolution-prop} and Remark \ref{real is R}, the hypotheses of Corollary \ref{cohomology generation} are satisfied. Theorem \ref{main thm} follows.\end{proof}

\begin{proof}[Proof of Theorem \ref{complex-oriented}]
By Theorem 7.3.5 of \cite{NakajimaJAMS}, $H^*(\Ms,\mathbb{Z})$ is known to be free abelian and concentrated in even degrees.  By the universal coefficient theorem, it follows that for any graded ring $E^*(\on{pt})$, $H^*(\Ms,\mathbb{Z})\otimes_{\mathbb{Z}} E^*(\on{pt}) = H^*\big(\Ms,E^*(\on{pt})\big)$ and
$H^*(B\mathbb{G},\mathbb{Z})\otimes_{\mathbb{Z}} E^*(\on{pt}) =
H^*(B\mathbb{G},E^*(\on{pt})\big)$.

The Atiyah-Hirzebruch spectral sequence for a cohomology theory $E$ and space $X$ has $E_2$-page
$E_2^{p,q} = H^p\big(X, E^q(\on{pt})\big) \implies E^{p+q}(X).$
By the previous paragraph, if $E^*(\on{pt})$ is evenly graded the spectral sequence degenerates at $E_2$ for both $E^*(\Ms)$ and $E^*(B\mathbb{G})$.  Assertion (1) of the theorem thus follows from Theorem \ref{main thm}.

To prove (2), we observe that all the ingredients of the proof of Proposition \ref{prop:resolution} hold in any complex-oriented cohomology theory $E$.  In particular, there is a Gysin map for proper morphisms and one can calculate $f^*$ via pull--cup-with-graph--push; that $[\Gamma] = c_d(R)$ and Chern classes of $R$ depend polynomially on the Chern classes of the tautological bundles follow from explicit formulas as in Lemmas 2.1 and 2.3 of \cite{Hudson}.   It remains to see that $E^*(\overline{\Ms})\rightarrow E^*(\Ms)$ is surjective; however, the natural map $\Ms\rightarrow B\mathbb{G}$ factors through $\overline{\Ms}\rightarrow B\mathbb{G}$ defined via projection of $\Ggtr$ on any factor $\mathbb{G}$, and surjectivity of $E^*(\overline{\Ms})\rightarrow E^*(\Ms)$ follows from that of $E^*(B\mathbb{G})\rightarrow E^*(\Ms)$.
\end{proof}

\begin{proof}[Proof of Theorem \ref{derived cat}]
We note that assertion (1) is immediate from assertion (2).  

In light of Remark \ref{real is R}, we will use the notation of Proposition \ref{prop:resolution} for the complex $R$.    
The Koszul complex associated to the complex $R$ and section $s$ of $\mathcal{H} = \mathcal{H}^0(R)$ of Theorem \ref{resolution-prop} provides a resolution (Section B.3.4 of \cite{Fulton}) of $\theo_{\Gamma}$,
\begin{equation}\label{Koszul}
\Big[\bigwedge^d \mathcal{H}^* \rightarrow \dots \rightarrow \bigwedge^2 \mathcal{H}^* \rightarrow \mathcal{H}^*\rightarrow \theo_{\Ms\times\overline{\Ms}}\Big] \simeq \theo_{\Gamma}.
\end{equation}
For each $k$, consider the $k$th tensor power $T^k(R)$ of the complex $R$: it is a differential graded vector bundle whose terms are tensor products of $\cE_j^\ell$s and $\cF_j^\ell$s.  The symmetric group $S_k$ naturally acts on $T^k(R)$ with the usual ${\mathbb Z}/2{\mathbb Z}$-graded sign conventions; we write $\bigwedge^k(R) = T^k(R)^{S_k, \on{sgn}}$, the sign-isotypic part of $T^k(R)$.  Both operations $T^k(-)$ and $(-)^{S_k,\on{sgn}}$ preserve quasi-isomorphism, hence $\bigwedge^k(R) \simeq \bigwedge^k(\mathcal{H})$.  The Koszul complex thus writes $\theo_\Gamma$ as an iterated cone on the complexes $\bigwedge^k(R)^\vee$.  

We remark that, viewing $\cE^\bullet:=\oplus_j\cE_j^\bullet$ and $\cF^\bullet:=\oplus_j\cF_j^\bullet$ as $\mathbb{Z}/2$-graded vector bundles, we find that 
$\bigwedge^k(R)$ is a direct summand of $\bigwedge^k(\cE^\bullet\boxtimes\cF^\bullet)$ in a canonical way.  Furthermore, following the work of \cite{BR}\footnote{We thank J. Weyman for help with references.} it is known that $\bigwedge^k(\cE^\bullet\boxtimes\cF^\bullet)$ is an iterated extension of tensor products of Schur functors applied to the $\mathbb{Z}/2$-graded vector bundles $\cE^\bullet$ and $\cF^\bullet$ (see Corollary 1.2 of \cite{EW} and the discussion preceeding it for more details).  Moreover, the expression for $\bigwedge^k(\cE^\bullet\boxtimes\cF^\bullet)$ as an iterated extension of $\mathcal{S}_{\lambda}(\cE^\bullet)$ and $\mathcal{S}_{\lambda}(\cF^\bullet)$ is compatible with the expression for $\bigwedge^k(R)$ as a direct summand of $\bigwedge^k(\cE^\bullet\boxtimes\cF^\bullet)$: in particular, $\bigwedge^k(R)$ is an iterated cone on external tensor products of the objects $S_{\lambda}(\cE_j^\ell)$, $S_\lambda(\cF_j^\ell)$ that are obtained by applying Schur functors to the various $\cE_j^\ell$ and $\cF_j^\ell$.  

Suppose  $\cG$ is a coherent complex on $\overline{\Ms}$. For any external tensor product $S_{\lambda}(\cE_j^\ell)^\vee\boxtimes N$, we have 
\bd
\mathbb{R}(p_{\Ms})_*\big((S_{\lambda}(\cE^\ell)^\vee\boxtimes N)\otimes (p_{\overline{\Ms}})^*\cG\big)
\simeq S_{\lambda}(\cE_j^\ell)^\vee \otimes U^\bullet
\ed
for some bounded complex $U^\bullet$ of finite-dimensional vector spaces.   
Using \eqref{Koszul} and the conclusion of the previous paragraph, we find that 
$\cG|_{\Ms}$ lies in the subcategory of $D_{\on{coh}}(\Ms)$ that is generated, under the operations (i)-(iii) of assertion (2) of Theorem \ref{derived cat}, by the $S_{\lambda}(\cE_j^\ell)^\vee$, where the Schur functors that appear are exactly those used in writing all the $\bigwedge^k(\mathcal{H})$ as above.
\end{proof}

\end{document}